\documentclass[final]{siamltex}

\usepackage{amsfonts}
\usepackage{latexsym}
\usepackage{amsmath, amssymb}
\usepackage{arydshln}
\usepackage{algorithm}
\usepackage{algpseudocode}
\usepackage[hyphens]{url}   % simple URL typesetting
\usepackage{hyperref} % hyperlinks
\usepackage{color}
\usepackage{multirow,multicol}
\usepackage{mathtools}
\usepackage{mathtools, eqparbox}

\newcommand{\varstackrel}[3][T]{\stackrel{\raisebox{0.5ex}{\clap{\scriptsize#2}}}{#3}}

\newcommand{\tsum}{\textstyle{\sum}}
\newcommand{\tprod}{\textstyle{\prod}}
\newcommand{\beq}{\begin{equation}}
\newcommand{\eeq}{\end{equation}}
\newcommand{\nn}{\nonumber}
\newcommand{\bbe}{\mathbb{E}}
\newcommand{\bbr}{\mathbb{R}}

\def\eqnok#1{(\ref{#1})}
\def\argmin{{\rm argmin}}

\def\Argmin{{\rm Argmin}}

\def\vgap{\vspace*{.1in}}
\def\prob{\mathop{\rm Prob}}
\def\Prob{{\hbox{\rm Prob}}}

\def\Pr{{{\cal M}}}
\def\la{{\langle}}
\def\ra{{\rangle}}
\def\0b{{\bf 0}}
\def\1b{{\bf 1}}
\def\SO{{\rm SFO}}

\newfloat{RGEM}{htbp}{pro}
\floatname{RGEM}{RGEM}
\makeatletter
\renewcommand{\fnum@RGEM}{\fname@RGEM}
\makeatother

\title{Random gradient extrapolation \\
for distributed and stochastic optimization}
\author{
	Guanghui Lan
	\thanks{H. Milton Stewart School of Industrial \& Systems Engineering, Georgia Institute of Technology, Atlanta, GA, 30332 .
		(email: {\tt george.lan@isye.gatech.edu}).}
	\and
	Yi Zhou
	\thanks{H. Milton Stewart School of Industrial \& Systems Engineering, Georgia Institute of Technology, Atlanta, GA, 30332 .
		(email: {\tt yizhou@gatech.edu}).}
}

\begin{document}
	
	\maketitle
	
	\begin{abstract}
		In this paper, we consider a class of finite-sum convex optimization problems defined over a distributed multiagent network with $m$ agents connected to a central server. In particular, the objective function consists of the average of $m$ ($\ge 1$) smooth components associated with each network agent together with a strongly convex term. 
		% We first introduce a deterministic predictive accelerated gradient descent (GEM) method that can achieve the optimal black-box iteration complexity 
		% for solving these convex optimization problems. 
		Our major contribution is to develop
		a new randomized incremental gradient algorithm, namely random gradient extrapolation method (RGEM), which does not require any exact gradient evaluation even for the initial point, but can achieve the optimal ${\cal O}(\log(1/\epsilon))$ complexity bound in terms of the total number of gradient evaluations of component functions to solve the finite-sum problems. 
		Furthermore, we demonstrate that for stochastic finite-sum optimization problems, RGEM maintains the optimal ${\cal O}(1/\epsilon)$ complexity (up to a certain logarithmic factor) in terms of the number of stochastic gradient computations, but attains an ${\cal O}(\log(1/\epsilon))$ complexity in terms of communication rounds (each round involves only one agent).
		It is worth noting that the former bound is independent of the number of agents $m$, while the latter one only linearly depends on $m$ or even $\sqrt m$ for ill-conditioned problems.
		 To the best of our knowledge, this is the first time that these complexity bounds have been obtained for distributed and stochastic optimization problems. 
		 Moreover, our algorithms were developed based on a novel dual perspective of Nesterov's accelerated gradient method.
		
		\vspace{.1in}
		
		\vspace{.1in}
	\end{abstract}
\vspace{0.1cm}

\setcounter{equation}{0}
\section{Introduction}
The main problem of interest in this paper is the finite-sum convex programming (CP) problem given in the form of
\beq\label{cp}
\psi^*:=\min_{x \in X} \left\{\psi(x):=\tfrac{1}{m}\tsum_{i=1}^{m}f_i(x)+\mu w(x)\right\}.
\eeq
Here, $X\subseteq \bbr^n$ is a closed convex set,
$f_i:X\rightarrow\bbr,\ \ i=1,\dots,m,$ 
are smooth convex functions
with Lipschitz continuous gradients over $X$, i.e., $\exists L_i\ge 0$ such that 
\beq\label{def_smoothness}
\|\nabla f_i(x_1)-\nabla f_i(x_2)\|_*\le L_i \|x_1-x_2\|, \ \ \forall x_1,x_2 \in X,
\eeq
$w: X\rightarrow\bbr$
is a strongly convex function with modulus $1$ w.r.t. a norm $\|\cdot\|$, i.e.,
\beq\label{def_strongconvexity}
w(x_1)-w(x_2)-\langle w'(x_2),x_1-x_2\rangle\ge \tfrac{1}{2}\|x_1-x_2\|^2, \ \ \forall x_1,x_2\in X,
\eeq
where $w'(\cdot)$ denotes any subgradient (or gradient) of $w(\cdot)$ and $\mu\ge 0$ is a given constant. Hence, the objective function $\psi$ is strongly convex whenever $\mu >0$. 
For notational convenience, we also denote 
$f(x) \equiv \tfrac{1}{m}\tsum_{i=1}^m f_i(x)$, 
$L \equiv \tfrac{1}{m}\tsum_{i=1}^m L_i$, and $\hat L=\max_{i=1,\dots,m}{L_i}$. It is easy to see that for some $L_f \ge 0$,
\beq \label{smooth_f}
\|\nabla f(x_1) - \nabla f(x_2) \|_* \le L_f \|x_1 - x_2\| \le L \|x_1- x_2\|, \ \ \forall x_1, x_2 \in X.
\eeq
We also consider a class of stochastic finite-sum optimization problems given by
\beq\label{sp}
 \psi^*:=\min_{x \in X} \left\{\psi(x):=\tfrac{1}{m}\tsum_{i=1}^{m}\bbe_{\xi_i}[F_i(x,\xi_i)]+\mu w(x)\right\},
 \eeq
where $\xi_i$'s are random variables with support $\Xi_i\subseteq \bbr^d$. It can be easily seen that \eqref{sp} is a special case of \eqref{cp} with $f_i(x)= \bbe_{\xi_i}[F_i(x,\xi_i)], \ i=1,\ldots,m$. However, different from deterministic finite-sum optimization problems, only noisy gradient information of each component function $f_i$ can be accessed for the stochastic finite-sum optimization problem in \eqref{sp}. 
% For example, any regularized empirical risk minimization problems of the form 
% \beq\label{ERM}
% \min_{x} \tfrac{1}{m}\tsum_{i=1}^m l_i(x) + \lambda R(x),
% \eeq
% where the loss function $l_i$ corresponding to the $i$-th dataset is smooth and convex, and the regularizer $R(\cdot)$ is strongly convex with modulus $1$ w.r.t. an arbitrary norm $\|\cdot\|$, belongs to the class of finite-sum optimization problem \eqref{cp}. 
% The stochastic finite-sum problems \eqref{sp} corresponds to the regularized minimization of general risk and is particularly useful for dealing with online (steaming) data distributed over a network. 

The deterministic finite-sum problem \eqref{cp} can model the empirical risk minimization in machine learning and statistical inferences, and hence has become the subject of intensive studies during the past few years.
Our study on finite-sum problems \eqref{cp} and \eqref{sp} has also been motivated by the emerging need for distributed optimization and machine learning. Under such settings, each component function $f_i$ is associated with an agent $i$, $i=1,\ldots,m$, which are connected through a distributed network. 
While different topologies can be considered for distributed optimization (see, e.g., Figure~\ref{example} and~\ref{dencentralized}),
% \todo{introduction to federated learning} 
% As oppose to the standard machine learning problems, which require centralizing the training data on one machine or in a data center, the notions of distributed machine learning and differential privacy have recently become the subject of interests. 
in this paper, we focus on the star network where $m$ agents are connected to one central server, and all agents only communicate with the server (see Figure~\ref{example}).
% In particular, each network agent $i$ is associated with the local component function $f_i$, and they work collaboratively to minimize the empirical (or general) risk using their private training data. The central server is in charge of controlling the entire machine learning process and updating outputs.  
These types of distributed optimization problems have several unique features. 
Firstly, they allow for data privacy, since no local data is stored in the server.  
Secondly, network agents behave independently and they may not be responsive at the same time.
% In fact, they may not be responsive at the same time and their responsive periods can not be determined in advance. 
Thirdly, the communication between the server and agent can be expensive and has high latency.
Finally, by considering the stochastic finite-sum optimization problem, we are interested in not only the deterministic empirical risk minimization, but also the generalization risk for distributed machine learning. Moreover, we allow the private data for each agent to be collected in an online (steaming) fashion.
One typical example of the aforementioned distributed problems is {\sl Federated Learning} recently introduced by Google in \cite{mcmahan2016communication}. 
As a particular example, in the $\ell_2$-regularized logistic regression problem, we have 
\[
f_i(x)=l_i(x):=\tfrac{1}{N_i}\tsum_{j=1}^{N_i}\log(1+{\rm exp}(-b_j^i{a_j^{i}}^Tx)), \ i=1,\ldots,m, \ \ w(x)= R(x):=\tfrac{1}{2}\|x\|^2_2,
\]
provided that $f_i$ is the loss function of agent $i$ with training data $\{a_j^i,b_j^i\}_{j=1}^{N_i}\in \bbr^n\times\{-1,1\}$, and $\mu:=\lambda$ is the penalty parameter. For minimization of the generalized risk, $f_i$'s are given in the form of expectation, i.e., 
\[
	f_i(x)=l_i(x):=\bbe_{\xi_i}[\log(1+{\rm exp}(-\xi_i^Tx))], \ i=1,\ldots,m,
\]
where the random variable $\xi_i$ models the underlying distribution for training dataset of agent $i$.
\begin{figure}[!htb]
   \begin{minipage}{0.48\textwidth}
     \centering
\footnotesize
\includegraphics[scale=0.5]{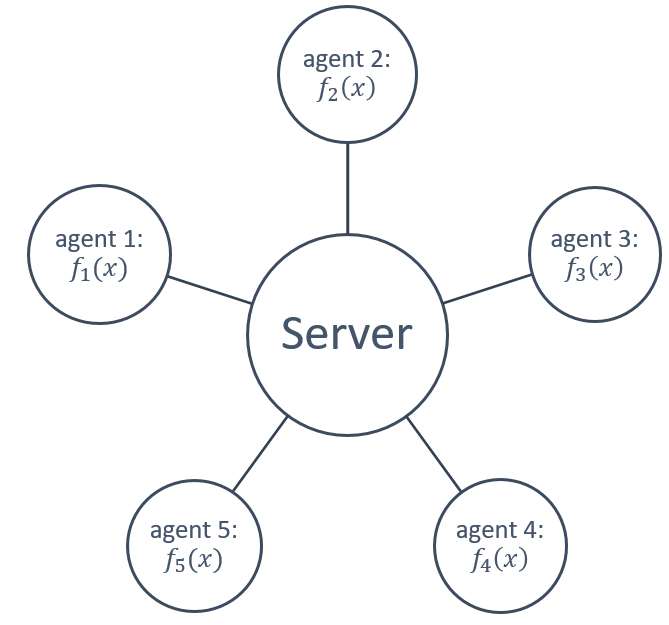}
\caption{A distributed network with $5$ agents and one server}
\label{example}
\end{minipage}\hfill
% \end{figure}
% \begin{figure}
\begin {minipage}{0.48\textwidth}
     \centering
\footnotesize
\includegraphics[scale=0.62]{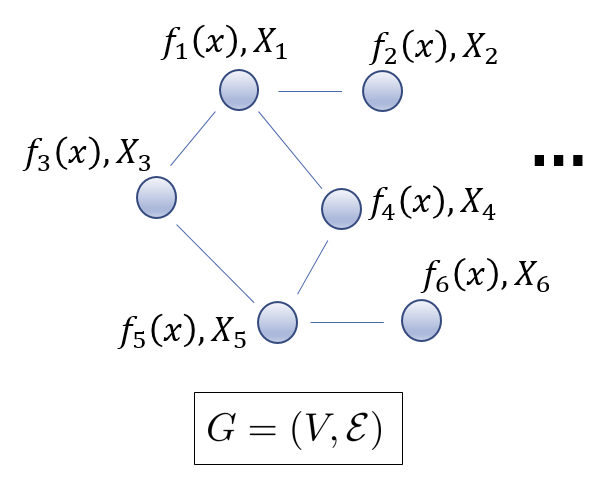}
\caption{An example of the decentralized network}
\label{dencentralized}
\end{minipage}
\end{figure}
% \end{multicols}
% namely the federated learning introduced by Google Research in \cite{mcmahan2016communication}.
% \todo{technical difficulties of this problem}
% The most straightforward way to solve the finite-sum optimization problems given in the form of \eqref{cp} is to apply the optimal deterministic first-order methods (FOMs), e.g., Nesterov's accelerated gradient method and its extensions \cite{Nest83-1,Nest04,Nest13-1,BecTeb09-2,tseng08-1,lan2015optimal}. 
% Although this type of methods achieves optimal iteration complexity for solving \eqref{cp}, it
% % designed for centralized smooth convex optimization problems and 
% requires the exact (full) gradient evaluation of $f$ iteratively, which is computationally expensive for distributed learning as it needs synchronization among $m$ devices across the network at each iteration. 
% \todo{briefly talks about decentralized case}
Note that another type of topology for distributed optimization is the multi-agent network without a central server, namely the decentralized setting, as shown in Figure~\ref{dencentralized},
where the agents can only communicate with their neighbors to  update information,
please refer to \cite{lan2017communication,shi2015extra,lee2017stochastic} and reference therein for decentralized algorithms.

% \subsection{Literature review}
% In order to fully exploit the special structure of the class of finite-sum optimization problems, researchers recently have utilized the randomized incremental gradient (RIG) algorithmic scheme \cite{Bertsekas10-1}, i.e., to update $x^t$ by
% \[
% x^t = x^{t-1}-\eta_t\nabla f_{i_t}(x^{t-1}),
% \]
% where $\eta_t$ is the stepsize given in advance, and $i_t$ is a random variable with support on $\{1,\ldots,m\}$.
% Note that comparing to the class of stochastic dual methods (e.g., \cite{ShaZhang15-1,ShaZhang13-1,Yuchen14}), each iteration of the RIG methods (e.g.,\cite{BlHeGa07-1,JohnsonZhang13-1,xiao2014proximal,DefBacLac14-1,SchRouBac13-1,lan2015optimal,allen2016katyusha,hazan2016variance,LinMaiHar15-1}) only involves the computation $\nabla f_i$, rather than solving a more complicated subproblem 
% \[
% \argmin\{\la g,y\ra +f_i^*(y) +\|y\|_*^2\},
% \]
% which sometimes may not have explicit solutions \cite{ShaZhang15-1} (e.g., the logistics regression problems). 
% \todo{
% The design of randomized incremental gradient methods to solve the finite-sum optimization problems \eqref{cp} has been extensively studied in recent years due to its broad applications in machine learning and statistical inference. 

During the past few years, 
randomized incremental gradient (RIG) methods have emerged as an important class of first-order methods for finite-sum optimization (e.g.,\cite{BlHeGa07-1,JohnsonZhang13-1,xiao2014proximal,DefBacLac14-1,SchRouBac13-1,lan2015optimal,allen2016katyusha,hazan2016variance,LinMaiHar15-1}).
 % in the past few years.
For solving nonsmooth finite-sum problems, Nemirovski et al. \cite{NJLS09-1,NYu78} showed that stochastic subgradient (mirror) descent methods can possibly save up to ${\cal O}(\sqrt m)$ subgradient evaluations. By utilizing the smoothness properties of the objective, Lan~\cite{Lan08} showed that one can separate the impact of variance from other deterministic components for stochastic gradient descent and  presented a new class of accelerated stochastic gradient descent methods to further improve these complexity bounds.
However, the overall rate of convergence of these stochastic methods  is still sublinear even for smooth and strongly finite-sum problems (see \cite{GhaLan12-2a,GhaLan13-1}). 
% Inspired by these results, 
% During the past few years, 
% randomized incremental gradient (RIG) methods have emerged as an important class of first-order methods for finite-sum optimization (e.g.,\cite{BlHeGa07-1,JohnsonZhang13-1,xiao2014proximal,DefBacLac14-1,SchRouBac13-1,lan2015optimal,allen2016katyusha,hazan2016variance,LinMaiHar15-1}) in the past few years.
% During the past few years, randomized incremental gradient (RIG) methods have emerged as an important class of first-order methods for finite-sum optimization (e.g.,\cite{BlHeGa07-1,JohnsonZhang13-1,xiao2014proximal,DefBacLac14-1,SchRouBac13-1,lan2015optimal,allen2016katyusha,hazan2016variance,LinMaiHar15-1}). 
% In particular, Blatt et al. \cite{BlHeGa07-1} proposed an incremental aggregated gradient (IAG) method to solve unconstrained strongly convex quadratic problems.
Inspired by these works and the success of the incremental aggregated gradient method by Blatt et al.\cite{BlHeGa07-1},  
% , which for $t\ge l$ updates $x^t$ as
% \[
% x^t = x^{t-1} -  \tfrac{\eta}{l}\tsum_{i=1}^l\nabla f_{(t-i)_t}(x^{t-i}).
% \]
% Instead of randomly sampling $i_t$, Blatt et al employed a cyclic choice of $i_t$, and proved IAG method achieved the linear rate of convergence (with inexplicit rate) for solving strongly convex quadratic problems. 
% \todo{Inspired by the work of Nemirovski et al. \cite{NJLS09-1,NYu78} that stochastic gradient descent methods (SGDs) utilizing unbiased estimators of gradients, i.e., mirror descent SA method, can exhibit optimal rate of convergence for strongly convex problems, researchers have been working on variance reduction technique of RIG methods.}
Schimidt et al. \cite{SchRouBac13-1} presented a stochastic average gradient (SAG) method, which uses randomized sampling of $f_i$ to update the gradients, 
% In particular, it updates $x^t$ according to
% \[
% x^t = x^{t-1}-\tfrac{\eta_t}{m}\tsum_{i=1}^my_i^{t-1}
% \]
% with 
% $y_{i_{t}}^t = \nabla f_{i_{t}}(x^{t})$ and $y_{i}^t=y_i^{t-1}, \ i\ne i_t$. 
and can achieve a linear rate of convergence, i.e., an ${\cal O}\left\{m + (mL/\mu)\log(1/\epsilon)\right\}$ complexity bound, to solve unconstrained finite-sum problems \eqref{cp}.
 % in terms of the termination criterion $\bbe[\|x^k-x^*\|^2_2]\le \epsilon$. 
% Inspired by SAG method, 
% several different groups of researchers proposed their own designs of RIG methods based on the concept of variance reduction. 
Johnson and Zhang later in \cite{JohnsonZhang13-1} presented a stochastic variance reduced gradient (SVRG) method, which computes an 
estimator of $\nabla f$ by iteratively updating the gradient of one randomly selected $f_i$ of the current exact gradient information 
and re-evaluating the exact gradient from time to time. Xiao and Zhang \cite{xiao2014proximal} later extended SVRG to solve proximal finite-sum problems \eqref{cp}. 
All these methods exhibit an improved ${\cal O}\left\{\left(m+ L/\mu\right)\log(1/\epsilon)\right\}$ complexity bound, and
Defazio et al. \cite{DefBacLac14-1} also presented
an improved SAG method, called SAGA, that can achieve such a complexity result. 
Comparing to the class of stochastic dual methods (e.g., \cite{ShaZhang15-1,ShaZhang13-1,Yuchen14}), each iteration of the RIG methods 
% (e.g.,\cite{BlHeGa07-1,JohnsonZhang13-1,xiao2014proximal,DefBacLac14-1,SchRouBac13-1,lan2015optimal,allen2016katyusha,hazan2016variance,LinMaiHar15-1}) 
only involves the computation $\nabla f_i$, rather than solving a more complicated subproblem 
\[
\argmin\{\la g,y\ra +f_i^*(y) +\|y\|_*^2\},
\]
which may not have explicit solutions \cite{ShaZhang15-1}.% (e.g., the logistics regression problems).

Noting that most of these RIG methods are not optimal even for $m=1$,
much recent research effort has been directed to the acceleration of RIG methods. 
% One natural acceleration idea for strongly convex problems is restarting technique. 
In 2015, Lan and Zhou in \cite{lan2015optimal} proposed a RIG method, namely randomized primal-dual gradient (RPDG) method, 
and show that its total number of gradient computations of $f_i$ can be bounded by
\beq\label{opt_complexity}
{\cal O}\left\{\left(m+\sqrt{\tfrac{mL}{\mu}}\right)\log\tfrac{1}{\epsilon}\right\}.
\eeq
The RPDG method utilizes a direct acceleration without even using the concept of variance reduction, evolving from 
the randomized primal-dual methods developed in \cite{Yuchen14,DangLan14-1} for solving saddle-point problems.
Lan and Zhou~\cite{lan2015optimal} also established a lower complexity bound for the RIG methods by showing that the number of gradient evaluations of $f_i$ required by any RIG methods
 to find an $\epsilon$-solution of \eqnok{cp}, i.e., a point $\bar x \in X$ s.t. 
$\bbe[\|\bar x - x^*\|^2_2] \le \epsilon$, cannot be smaller than 
\beq \label{RIG_lb}
{\Omega} \left( \left(m + \sqrt{\tfrac{m L}{\mu}}\right) \log \tfrac{1}{\epsilon}\right),
\eeq
whenever the dimension 
\[
n\ge (k+m/2)/\log(1/q),
\] 
where $k$ is the total number of iterations and $q = 1 + 2/(\sqrt{L/((m+1)\mu)}-1)$. 
% Shalev-Shwartz and Zhang in \cite{ShaZhang15-1} proposed Accelerated Prox-SDCA method which achieved the same complexity result as in \eqref{opt_complexity_sc}. 
Simultaneously, Lin et al.~\cite{LinMaiHar15-1} presented a catalyst scheme which utilizes a restarting technique to
accelerate the SAG method in \cite{SchRouBac13-1} (or other ``non-accelerated'' first-order methods) and
thus can possibly improve the complexity bounds obtained by SVRG and SAGA to 
\eqref{opt_complexity}
(under the Euclidean setting). 
% In 2015, Lin, Mairal, and Harchaoui~\cite{LinMaiHar15-1} presented a catalyst scheme utilizing restarting technique that can be used to accelerate the SAG method in \cite{SchRouBac13-1} (or other ``non-accelerated'' first-order methods) and
% thus possibly improved the complexity bounds obtained by SVRG and SAGA to 
% \beq\label{opt_complexity}
% {\cal O}\left\{\left(m+\sqrt{\tfrac{mL}{\mu}}\right)\log\tfrac{1}{\epsilon}\right\},
% \eeq
% under the Euclidean setting. 
% % Shalev-Shwartz and Zhang in \cite{ShaZhang15-1} proposed Accelerated Prox-SDCA method which achieved the same complexity result as in \eqref{opt_complexity_sc}. 
% Simultaneously, Lan and Zhou in \cite{lan2015optimal} proposed a RIG method with built-in acceleration scheme, namely randomized primal-dual gradient (RPDG) method, and show that its total number of gradient computations of $f_i$ can be bounded by \eqref{opt_complexity}. 
% Clearly, RPDG method can be easily implemented since it does not needs to restart frequently.  
Allen-Zhu \cite{allen2016katyusha} later showed that one can also directly accelerate SVRG to achieve the optimal rate of convergence \eqref{opt_complexity}.
All these accelerated RIG methods can save up to ${\cal O}(\sqrt m)$ in the number of gradient evaluations of $f_i$ comparing to optimal deterministic first-order methods when $L/\mu\ge m$. 
% for similar complexity result obtained by accelerated SVRG.
% By comparing \eqref{opt_complexity} with the above lower complexity bound, we conclude that the aforementioned convergence rate obtained by catalyst scheme, RPDG and accelerated SVRG are not improvable. 

It should be noted that most existing RIG methods were inspired by empirical risk minimization on a single server (or cluster) in machine learning rather than on a set of agents distributed over a network. Under the distributed setting, methods requiring full gradient computation and/or restarting from time to time may incur extra communication and synchronization costs. As a consequence, methods which require fewer full gradient computations (e.g. SAG, SAGA and RPDG) seem to be more advantageous in this regard. An interesting but yet unresolved question in stochastic optimization is whether there exists a method which does not require the computation of any full gradients (even at the initial point), but can still achieve the optimal rate of convergence in \eqref{opt_complexity}. 
Moreover, little attention in the study of RIG methods has been paid to the stochastic finite-sum problem in \eqref{sp}, which is important for generalization risk minimization in machine learning. Very recently, there are some progresses on stochastic primal-dual type methods for solving problem \eqref{sp}. For example, Lan, Lee and Zhou \cite{lan2017communication} proposed 
a stochastic decentralized communication sliding method that can achieve
the optimal sampling complexity of ${\cal O}(1/\epsilon)$ and best-known ${\cal O}(1/\sqrt\epsilon)$ complexity bounds for communication rounds for solving stochastic decentralized strongly convex problems.
For the distributed setting with a central sever, by using mini-batch technique to collect gradient information and any stochastic gradient based algorithm as a black box to update iterates, 
Dekel et al. \cite{dekel2012optimal} presented a distributed mini-batch algorithm with a batch size of ${o}(m^{1/2})$ that can obtain ${\cal O}(1/\epsilon)$ sampling complexity (i.e., number of stochastic gradients) for stochastic strongly convex problems, and hence implies at least ${\cal O}(1/\sqrt \epsilon)$ bound for communication complexity.
An asynchronous version was later proposed by Feyzmahdavian et al. in \cite{feyzmahdavian2016asynchronous} that maintained the above convergence rate for regularized stochastic strongly convex problems. 
It should be pointed out that these mini-batch based distributed algorithms require sampling from all network agents iteratively and hence leads to at least ${\cal O}(m/\sqrt \epsilon)$ rate of convergence in terms of communication costs among server and agents.
% \todo{Moreover, these distributed methods exploits the schemes of stochastic gradient descent (SGD) methods, since they possess low iteration cost and can achieve optimal sampling complexity. However, under distributed settings, SGDs usually incur high communication costs due to their slow iteration complexities.}
It is unknown whether there exists an algorithm which only requires a significantly smaller communication rounds (e.g. ${\cal O}(\log 1/\epsilon)$), but can achieve the optimal ${\cal O}(1/\epsilon)$ sampling complexity for solving the stochastic finite-sum problem in \eqref{sp}.

% To the best of our knowledge, none of existing algorithm for distributed stochastic optimization problems can achieve a linear rate of convergence in terms of communication rounds.
% }

% Most recent advances in RIG methods to solve \eqref{cp} is the randomized primal-dual gradient (RPDG) method proposed by Lan and Zhou in \cite{lan2015optimal}. Meanwhile, a catalyst scheme is presented by Lin et al.~\cite{LinMaiHar15-1}, and later accelerated SVRG by Allen-Zhu in \cite{allen2016katyusha}.
% All three methods can achieve 
% \beq\label{opt_complexity}
% {\cal O}\left\{\left(m+\sqrt{\tfrac{mL}{\mu}}\right)\log\tfrac{1}{\epsilon}\right\},
% \eeq
% rate of convergence for solving \eqref{cp}.
% It should be noted that Lan and Zhou in \cite{lan2015optimal} also established a lower complexity bound on the number of gradient evaluations of $f_i$ required by any RIG methods to find an $\epsilon$-solution of \eqnok{cp}, i.e., a point $\bar x \in X$ s.t. 
% $\bbe[\|\bar x - x^*\|^2_2] \le \epsilon$, cannot be smaller than 
% \beq \label{RIG_lb}
% {\Omega} \left( \left(m + \sqrt{\tfrac{m L}{\mu}}\right) \log \tfrac{1}{\epsilon}\right),
% \eeq
% whenever the dimension $n\ge (k+m/2)/\log(1/q)$ with $k$ being the total iterations and $q = 1 + \frac{2}{\sqrt{L/((m+1)\mu)}-1}$.  
% Hence, the above methods are optimal to solve \eqref{cp}. 

% \todo{highlight contributions? rewrite short!}
The main contribution of this paper is to introduce a new randomized incremental gradient type method
% , namely the randomized gradient extrapolation method (RGEM), 
to solve \eqref{cp} and \eqref{sp}.  
Firstly, we develop a random gradient extrapolation method (RGEM) for solving \eqref{cp} that does not require any exact gradient evaluations of $f$. 
For strongly convex problems, we demonstrate that RGEM can still achieve the optimal rate of convergence \eqref{opt_complexity} under the assumption that the average of gradients of $f_i$ at the initial point $x^0$ is bounded by $\sigma_0^2$. 
% Additionally, RGEM relaxes RPDG's restricted differentiability assumption and hence can be applied to solve a broader class of problems. It also possesses simpler algorithmic scheme and convergence analysis than RPDG. 
% In addition, if the feasible set $X$ is bounded with diameter $D_X$ defined in \eqref{def_DX}, RGEM method maintains the same order of convergence rate as \eqref{complexity_rpaged_sm}.  
To the best of our knowledge, this is the first time that such an optimal RIG methods without any exact gradient evaluations 
has been presented for solving \eqref{cp} in the literature.
In fact, without any full gradient computation, RGEM possesses iteration costs as low as pure stochastic gradient descent (SGD) methods, 
but achieves a much faster and optimal linear rate of convergence for solving deterministic finite-sum problems. 
In comparison with the well-known randomized Kaczmarz method \cite{strohmer2009randomized}, which can be
viewed as an enhanced version of SGD, but can achieve a linear rate of convergence for solving linear systems, RGEM has a better convergence 
rate in terms of the dependence on the condition number $L/\mu$.
Secondly, we develop a stochastic version of RGEM and establish its optimal convergence properties for solving stochastic finite-sum problems \eqref{sp}. 
More specifically, we assume that only noisy first-order information of one randomly selected component function $f_i$ can be accessed via a stochastic first-order ($\SO$) oracle iteratively. 
In other words, at each iteration only one randomly selected network agent needs to compute an estimator of its gradient by sampling from its local data using a $\SO$ oracle instead of performing exact gradient evaluation of its component function $f_i$.
Note that for these problems, it is difficult to compute the exact gradients even at the initial point.
Under standard assumptions for centralized stochastic optimization, i.e., the gradient estimators computed by the $\SO$ oracle are unbiased and have bounded variance $\sigma^2$, the number of stochastic gradient evaluations performed by RGEM to solve \eqref{sp} can be bounded by\footnote{$\cal \tilde O$ indicates the rate of convergence is up to a logarithmic factor - $\log(1/\epsilon)$.}
\beq\label{complexity_sfo}
{\cal \tilde O}
\left\{\tfrac{\sigma_0^2/m+\sigma^2}{\mu^2\epsilon}+\tfrac{\mu \|x^0-x^*\|^2_2 +\psi(x^0)-\psi^* }{\mu\epsilon}\right\},
\eeq
for finding a point $\bar x \in X$ s.t. $\bbe[\|\bar x-x^*\|_2^2] \le \epsilon$.
% with $P(\cdot,\cdot)$ being a generalized Bregman distance defined later in \eqref{primal_prox}. 
Moreover, by utilizing the mini-batch technique, RGEM can achieve an
\beq\label{complexity_sfocom}
	{\cal O} \left\{\left(m +\sqrt{\tfrac{m\hat L}{\mu}}\right)\log \tfrac{1}{\epsilon}\right\},
\eeq
complexity bound in terms of the number of  communication rounds, and each round only involves the communication between the server and a randomly selected agent. 
This bound seems to be optimal, since it matches the lower complexity bound for RIG methods to solve deterministic finite-sum problems.
It is worth noting that the former bound \eqref{complexity_sfo} is independent of the number of agents $m$, while the latter one \eqref{complexity_sfocom} only linearly depends on $m$ or even $\sqrt m$ for ill-conditioned problems.
To the best of our knowledge, this is the first time that such a RIG type method has been developed for solving stochastic finite-sum problems \eqref{sp}
that can achieve the optimal communication complexity and nearly optimal (up to a logarithmic factor) sampling complexity in the literature. 
% \todo{add paragraph here!}

RGEM is developed based on a novel algorithmic framework, namely gradient extrapolation method (GEM), that we introduce in this paper for solving black-box convex optimization (i.e., $m=1$).
The development of GEM was inspired by our recent studies on the relation between accelerated gradient methods and the primal-dual gradient methods. 
In particular, it is observed in \cite{lan2015optimal} that Nesterov's accelerated gradient method is a special primal-dual gradient (PDG) method where the extrapolation step is performed in the primal space. Such a primal extrapolation step, however, might result in a search point outside the feasible region under the randomized setting in the RPDG method mentioned above. In view of this deficiency of PDG and RPDG methods, we propose to switch the primal and dual spaces for primal-dual gradient methods, and to perform the extrapolation step in the dual (gradient) space. The resulting new first-order method, i.e., GEM, can be viewed as a dual version of Nesterov's accelerated gradient method, and we show that it can also achieve the optimal rate of convergence for black-box convex optimization.

RGEM is a randomized version of GEM which only computes the gradient of a randomly selected component function $f_i$ at each iteration. It utilizes the 
gradient extrapolation step also for estimating
exact gradients in addition to predicting dual information as in GEM. 
As a result, it has several advantages over RPDG.
Firstly, RPDG requires a restricted assumption that each $f_i$ has to be differentiable and has Lipschitz continuous gradients over the whole $\bbr^n$ due to its primal extrapolation step.
RGEM relaxes this assumption to having Lipschitz gradients over the feasible set $X$ (see \eqref{def_smoothness}), and hence can be applied to a much broader class of problems. 
% RPDG only requires a full gradient computation at the initial point. However, it has a restrictive assumption that each $f_i$ has to be differentiable and has Lipschitz continuous gradients over the whole $\bbr^n$. 
% instead of the feasible set $X$ due to its primal extrapolation step.
Secondly, RGEM possesses simpler convergence analysis carried out in the primal space due to its simplified algorithmic scheme. However, RPDG has a complicated algorithmic scheme, which contains a primal extrapolation step and a gradient (dual) prediction step in addition to solving a primal proximal subproblem, and thus leads to an intricate primal-dual convergence analysis. 
% while the proposed RGEM method relaxes it to $f_i$ being differentiable over the feasible set $X$. 
Last but not least, it is unknown whether RPDG could maintain the optimal convergence rate \eqref{opt_complexity} without the exact gradient evaluation of $f$ during initialization. 
% }

%  which only requires one gradient extrapolation step iteratively in addition to gradient evaluation of the randomly selected component function $f_i$ to estimate the future exact gradient information of $f$. 
% It leads to a simple convergence analysis of RGEM carrier out completely in the primal space.
% Moreover, it is worth mentioning that the idea of designing such gradient extrapolation based algorithmic scheme is inspired by the observation that the well-known Nesterov's accelerated gradient method is a special case of the primal-dual gradient method \cite{lan2015optimal}. More specifically, we demonstrate that when $m=1$ RGEM (reduced to GEM in Algorithm~\ref{alg_paged}) can be viewed as a dual of Nesterov's accelerated gradient method in Section~\ref{sec_deter}.

This paper is organized as follows. In Section~\ref{sec_results} we present the proposed random gradient extrapolation methods (RGEM), and their convergence properties for solving \eqref{cp} and \eqref{sp}.
In order to provide more insights into the design of the algorithmic scheme of RGEM, we provide an introduction to the gradient extrapolation method (GEM) and its relation to the primal-dual gradient method, as well as Nesterov's method in Section~\ref{sec_deter}. Section~\ref{sec_conv} is devoted to the convergence analysis of RGEM. Some concluding remarks are made in Section 5.

\subsection{Notation and terminology}
We use $\|\cdot\|$ to denote a general norm in $\bbr^n$ without specific mention.
% , which is not necessary associated the inner product $\langle\cdot,\cdot\rangle$.
We also use $\|\cdot\|_*$ to denote the conjugate norm of $\|\cdot\|$. 
For any $p\ge 1$, $\|\cdot\|_p$ denotes the standard $p$-norm in $\bbr^n$, i.e.,
$
\|x\|^p_p=\tsum_{i=1}^{n}|x_i|^p, \  \mbox{for any } x\in \bbr^n.
$
% We define the diameter of a convex set $X$ as 
% \beq \label{def_DX} 
% D_X \equiv
% D_{X,\|\cdot\|}:= \max_{x, y \in X} \|x - y\|. 
% \eeq
For any convex function $h$, $\partial h(x)$ is the set of subdifferential at $x$. 
% Given any $X\subseteq \bbr^n$, we say a convex function $f: X\rightarrow\bbr$ is smooth if it is Lipschitz continuously differentiable with Lipschitz constant $L>0$, i.e., 
% $
% \|\nabla f(x)-\nabla f(y)\|_*\le L\|x-y\|, \ \mbox{for any }x,y\in X.
% $
% We say that a convex function $h: X\rightarrow\bbr$ is strongly convex with modulus $\mu$ with respect to $\|\cdot\|$, if it satisfies
% $h(y)\ge h(x)+\langle h'(x), y-x\rangle+\frac{\mu}{2}\|x-y\|^2, \ \ \mbox{for any }x,y\in X,
% $
% where $h'(x)\in \partial h(x)$ is an arbitrary subgradient of $h$ at $x$.
% Recall that the function $w(\cdot)$ in \eqref{cp} is strongly convex with modulus $1$ w.r.t. $\|\cdot\|$. 
For a given strongly convex function $w$ with modulus $1$ (see \eqref{cp}), we define a {\sl prox-function} associated with $w$ as
\beq\label{primal_prox}
P(x^0,x)\equiv P_{w}(x^0,x):=w(x)-\left[w(x^0)+\langle w'(x^0),x-x^0\rangle\right],
\eeq
where $w'(x^0)\in \partial w(x^0)$ is an arbitrary subgradient of $w$ at $x^0$. By the strong convexity of $w$,
we have 
\beq\label{P_strong}
P(x^0,x)\ge \frac{1}{2}\|x-x^0\|^2, \ \ \forall x,x^0\in X.
\eeq
It should be pointed out that the prox-function $P(\cdot,\cdot)$ described above is a generalized Bregman distance in the sense that $w$ is not necessarily differentiable. 
This is different from the standard definition for Bregman distance~\cite{Breg67,AuTe06-1,BBC03-1,Kiw97-1,censor1981iterative}. 
%\todo{I cited his 1981 paper(please cite to the paper for one Israel professor who contact us)}.
Throughout this paper, we assume that the prox-mapping associated with $X$ and $w$, given by
\beq\label{prox_mapping}
\mathcal{M}_{X}(g,x^0,\eta):=\argmin_{x\in X}\left\{\langle g,x\rangle+ \mu w(x)+\eta P(x^0,x) \right\},
\eeq
is easily computable for any $x^0\in X, g\in \bbr^{n}, \mu\ge 0, \eta>0$. 
For any real number $r$, $\lceil r \rceil$ and $\lfloor r \rfloor$ denote the nearest integer to
$r$ from above and below, respectively. $\bbr_+$ and $\bbr_{++}$, respectively, denote the set of nonnegative 
and positive real numbers.

% \subsection{Organization of the paper}
% \todo{
% This paper is organized as follows. In Section~\ref{sec_results} we present the main results of this paper, which are our proposed RGEM algorithms, as well as their convergence results.
% In order to provide more insights into the design of the algorithmic scheme of RGEM and introduce our convergence analysis in a tutorial fashion, we first study the deterministic gradient extrapolation method (GEM) in Section~\ref{sec_deter}. For readers who feel comfortable with randomized methods, Section~\ref{sec_conv} is devoted to the convergence analysis of randomized gradient extrapolation method (RGEM). 
% Detailed convergence analysis of distributed RGEM and online distributed RGEM will be discussed in Section~\ref{sec_rws} and \ref{sec_sto} respectively.
% }

\section{Algorithms and main results}\label{sec_results}
% \todo{needs edits!}
This section contains three subsections.
We first present in Subsection~\ref{sec_alg} an optimal random gradient extrapolation method (RGEM) for solving the distributed finite-sum problem in \eqref{cp}, and then discuss in Subsection~\ref{sec_online}, a stochastic version of RGEM for solving the stochastic finite-sum problem in \eqref{sp}.
Subsection~\ref{sec_distributed} is devoted to the implementation of RGEM in a distributed setting and the discussion about its communication complexity.

\subsection{RGEM for deterministic finite-sum optimization}\label{sec_alg}
The basic scheme of RGEM is formally stated in Algorithm~\ref{alg_rpaged}. 
This algorithm 
% allows two different ways to initialize the gradient $y^0$. 
% In Type I initialization, an exact gradient evaluation of $f$ at the initial point $x^0$ is required,while for Type II initialization, it 
simply initializes the gradient as $y^{-1}= y^{0}= \0b$.
At each iteration,
RGEM requires the new gradient information of only one randomly selected component function $f_i$, but maintains $m$ pairs of search points and gradients 
$(\underline{x}_i^t, y_i^t)$, $i = 1, \ldots, m$, which are stored by their corresponding agents in the distributed network. 
More specifically, it first performs a gradient extrapolation step in \eqref{def_tildeyt} and the primal proximal mapping in \eqref{def_xt1}. Then 
a randomly selected block $\underline{x}_{i_t}^t$ is updated in \eqref{def_underlx} and  
the corresponding component gradient $\nabla f_{i_t}$ is computed in \eqref{def_yt1}. 
As can be seen from Algorithm~\ref{alg_rpaged}, 
% if Type II initialization is used, 
RGEM does not require any exact gradient evaluations.
% These two update steps are performed by the randomly chosen network agent $i_t$.
% It should be pointed out that RGEM allows two ways to initialize $y^0$. In Case I, an exact gradient evaluation of $f$ at the initial point $x^0$ is required. For Case II, RGEM skips this one-time exact gradient evaluation step, and it, without loss of generosity, set $y^{-1}= y^{0}= \0b$.
% In the process of the algorithm,
% RGEM requires the gradient evaluation of only one randomly selected component function $f_i$ at each iteration. 
% Hence little synchronization is required in the aforementioned distributed network setting. 
% An immediate improvement of RGEM with Case II initialization over any other RIG algorithms, like (accelerated) SVRG and SAGA, is that it does not incur synchronous delays under our proposed distributed setting, since no exact gradient evaluations of $f$ is required.
% Since it starts with $y_i^0=\nabla f_i(\underline x_i^0), \ i=1,\ldots,m$, and only updates the corresponding $i_t$-block of $(\underline{x}_i^t, y_i)$, $i=1,\ldots,m$, at the $t$-th iteration, we always have 
% \beq\label{yt_def}
% y_i^t=\nabla f_i(\underline x_i^t), \ i=1,\ldots,m.
% \eeq  
% Observe that this method reduces to GEM whenever $m=1$. 
%Secondly,  are equivalent to the predictive step \eqref{def_tgt_i}, which computes an estimator $\tilde y^t$ of the future gradient using historic data. Since RGEM method requires less gradient information per iteration, the predictive steps only use gradient information of the updated block. 

\begin{algorithm}
	\caption{A random gradient extrapolation method (RGEM)}
	\label{alg_rpaged}
	\begin{algorithmic}
		\State {\bf Input:}
		Let $x^0\in X$, and the nonnegative parameters $\{\alpha_t\}$, $\{\eta_t\}$, and $\{\tau_t\}$ be given.
		%		\State {\bf Warm start stage:} $\bar x= Algorithm\ref{alg_spdg}(x_0,d,\{\tilde \tau_t\},\{\tilde \eta_t\},\{\tilde \alpha_t\})$.
		\State{\bf Initialization:}
		\State
		Set $\underline{x}_{i}^0=x^0, \ i=1,\ldots,m$,
		% \State {\bf Type I}: $y^{-1}=y^0=(\nabla f_1(x^0),\dots, \nabla f_m(x^0))$. \Comment{One-time exact gradient evaluation for initialization}
		% \State {\bf Type II}:
		% \State 
		$y^{-1}=y^{0}=\0b$.
		\Comment{No exact gradient evaluation for initialization} 
		%and $g^0=\tfrac{1}{m}\tsum_{i=1}^my_i^0$.
		\State
		{\bf for} $t=1,\dots,k$ {\bf do}
		
		Choose $i_t$ according to $\Prob\{i_t=i\}=\tfrac{1}{m}, i=1,\dots,m$.
		
		Update $z^t=(x^t,y^t)$ according to 
		\begin{align}
		\tilde{y}^t&=y^{t-1} + \alpha_t(y^{t-1}-y^{t-2}),\label{def_tildeyt}\\
		%		\tilde{g}^t&=g^{t-1}+\tfrac{\alpha_t}{m}(y_{i_{t-1}}^{t-1}-y_{i_{t-1}}^{t-2}),\label{def_tildeg}\\
		x^t&=\mathcal{M}_{X}(\tfrac{1}{m}\tsum_{i=1}^m\tilde y^t_i,x^{t-1},\eta_t),\label{def_xt1}\\
		\underline{x}_{i}^t&=
		\begin{cases}
		(1+\tau_t)^{-1}(x^t+\tau_t\underline{x}^{t-1}_{i}), \ \ i=i_t,\\
		\underline{x}_{i}^{t-1}, \ \ i\ne i_t.
		\end{cases}\label{def_underlx}\\
		y_i^t&=
		\begin{cases}
		\nabla f_i(\underline{x}_{i}^t), \ \ i=i_t,\\
		y_i^{t-1}, \ \ i\ne i_t.
		\end{cases}\label{def_yt1}
		%		g^t&=g^{t-1}+\tfrac{1}{m}(y_{i_t}^{t}-y_{i_t}^{t-1}), \label{def_gt_r}
		\end{align}  
		\State {\bf end for}
		\State {\bf Output:} For some $\theta_t>0$, $t=1,\ldots,k$, set
		\beq\label{def_ergodic_m}
		\underline x^k:=(\tsum_{t=1}^k\theta_t)^{-1}\tsum_{t=1}^k\theta_t x^t.
\eeq
	\end{algorithmic}
\end{algorithm}

% We add one remark about the implementation of Algorithm~\ref{alg_rpaged}.
Note that the computation of $x^t$ in \eqref{def_xt1} requires an involved computation of $\tfrac{1}{m}\tsum_{i=1}^m \tilde y_i^t$. In order to save computational time when implementing this algorithm,
we suggest to compute this quantity in a recursive manner as follows. Let us denote $g^t \equiv \tfrac{1}{m} \tsum_{i=1}^m y_i^t, \ t = 1,\ldots,k$.
Clearly, in view of the fact that $y^t_i = y^{t-1}_i$, $\forall i \neq i_t$, we have
\beq\label{compute_gt}
g^t = g^{t-1} + \tfrac{1}{m}(y^t_{i_t} - y^{t-1}_{i_t}).
\eeq
Also, by the definition of $g^t$ and \eqref{def_tildeyt}, we have
\begin{align}\label{compute_tildeyt}
\tfrac{1}{m}\tsum_{i=1}^m \tilde y^t_i 
&= \tfrac{1}{m}\tsum_{i=1}^m y_i^{t-1} + \tfrac{\alpha_t}{m} (y^{t-1}_{i_{t-1}} - y^{t-2}_{i_{t-1}})  
=g^{t-1} + \tfrac{\alpha_t}{m}(y^{t-1}_{i_{t-1}} - y^{t-2}_{i_{t-1}}).
\end{align}
Using these two ideas mentioned above, we can compute $\tfrac{1}{m}\tsum_{i=1}^m \tilde y_i^t$ in two steps:  
i) initialize $g^0 = \0b$,
% \[
% g^0 = \begin{cases}
% \tfrac{1}{m}\tsum_{i=1}^m f_i(x^0), \ &\text{for Type I initialization}, \\ 
% \0b, \ &\text{for Type II initialization},
% \end{cases}
% \] 
and update $g^t$ as in \eqref{compute_gt} after the gradient evaluation step \eqref{def_yt1};
ii) replace \eqref{def_tildeyt} by \eqref{compute_tildeyt} to compute $\tfrac{1}{m}\tsum_{i=1}^m \tilde y_i^t$.
Also note that the difference  $y^t_{i_t} - y^{t-1}_{i_t}$ can be saved as it is used in both \eqnok{compute_gt}
and \eqnok{compute_tildeyt} for the next iteration. 
These enhancements will be incorporated into the distributed setting in Subsection~\ref{sec_distributed} to possibly save communication costs.

It is also interesting to observe the differences between RGEM and RPDG \cite{lan2015optimal}. 
RGEM has only one extrapolation step \eqref{def_tildeyt} which combines two types of predictions. 
One is to predict future gradients using historic data, and the other is to obtain an estimator of the current exact gradient of $f$ from the randomly updated gradient information of $f_i$.
% As shown later in Corollaries~\ref{main_ran_sc} and~\ref{main_ran_sm},
%the selection of the extrapolation parameter $\alpha_t$ also reveals this underlying idea. While the primal-dual type methods usually set $\alpha_t={\cal O}(1)$, RGEM 
%chooses $\alpha_t= {\cal O}(m)$ as a reflection of the selection of uniform distribution $\Prob\{i_t=i\}=1/m, i=1,\dots,m$. 
However, RPDG method needs two extrapolation steps in both the primal and dual spaces. 
Due to the existence of the primal extrapolation step, RPDG cannot guarantee the search points where it performs gradient evaluations to fall within the feasible set $X$. Hence,
it requires the assumption that $f_i$'s are differentiable with Lipschitz continuous gradients over $\bbr^n$. Such a strong assumption is
not required by RGEM, since all the primal iterates generated by RGEM stay within the feasible region $X$. As a result, RGEM can deal with a much wider class of problems than RPDG. 
Moreover, RGEM allows no exact gradient computation for initialization, which provides a fully-distributed algorithmic framework under the assumption that there exists $\sigma_0\ge 0$ such that 
\beq\label{def_sigma}
\tfrac{1}{m}\tsum_{i=1}^m\|\nabla f_i(x^0)\|_*^2\le \sigma_0^2,
\eeq
where $x^0$ is the given initial point.  

We now provide a constant step-size policy for RGEM to solve strongly convex problems given in the form of \eqref{cp} and show that the resulting algorithm exhibits an optimal linear rate of convergence in Theorem~\ref{main_ran_sc}.
% With proper choices of parameters shown in the following corollary, the above distributed RGEM obtains an optimal linear rate of convergence for solving strongly convex problems in the form of \eqref{cp}. 
The proof of Theorem~\ref{main_ran_sc}
% .a) and Theorem~\ref{main_ran_sc}.b) 
can be found in Subsection~\ref{sec_rws}.
 % respectively.

\vgap
\begin{theorem}\label{main_ran_sc}
	Let $x^*$ be an optimal solution of \eqnok{cp}, $x^k$ and $\underline x^k$ be defined in \eqref{def_xt1} and \eqnok{def_ergodic_m}, respectively, and $\hat L=\max_{i=1,\dots,m}{L_i}$.
	Also let $\{\tau_t\}$, $\{\eta_t\}$ and $\{\alpha_t\}$ be set to 
	\begin{align} \label{constant_step_s_ran}
	\tau_t \equiv \tau=\tfrac{ 1}{m(1-\alpha)}-1, \ \ \ \eta_t \equiv \eta= \tfrac{\alpha}{1-\alpha}\mu, \ \ \ \mbox{and} \ \ \ \alpha_t \equiv m\alpha.
	\end{align} 
	% \tau_t \equiv \tau=\tfrac{ 1}{m(1-\alpha)}-1, \ \ \ \eta_t \equiv \eta= \tfrac{\alpha}{1-\alpha}\mu, \ \ \ \mbox{and} \ \ \ \alpha_t \equiv m\alpha
	% \end{align}
	% \begin{itemize}
	% 	\item [a)] If Type I initialization is used, and 
	% 	\beq\label{def_alpha}
	% 	\alpha=1-\tfrac{2}{m+\sqrt{m^2+8m\hat L/\mu}},
	% 	\eeq
	% 	then
	% 	\begin{align}
	% \bbe[P(x^k,x^*)]&\le \tfrac{2 \Delta_0\alpha^k}{\mu}, \label{ran_bnd_s1}\\
	% \bbe[\psi(\underline{x}^k)-\psi(x^*)]&\le 6\max\left\{m,\tfrac{\hat L}{\mu}\right\}\Delta_0 \alpha^{k/2}, \label{ran_bnd_s2}
	% \end{align} 
	% where 
	% \beq\label{def_Delta}
	% \Delta_0 := \mu P(x^0,x^*)+\psi(x^0)-\psi^*.
	% \eeq
	% with $\Delta_0$ defined in \eqref{def_Delta}.
	% \item [b)] If 
	% Type II initialization is used, 
	If \eqref{def_sigma} holds and $\alpha$ is set as 
		\beq\label{def_nalpha}
		\alpha=1-\tfrac{1}{m+\sqrt{m^2+16m\hat L/\mu}},
		\eeq
		then 
		\begin{align}
		\bbe[P(x^k,x^*)]&\le \tfrac{2\Delta_{0,\sigma_0}\alpha^k}{\mu},\label{ran_bnd_s1_rws}\\
		\bbe[\psi(\underline{x}^k)-\psi(x^*)]&\le 6\max\left\{m,\tfrac{\hat L}{\mu}\right\}\Delta_{0,\sigma_0}\alpha^{k/2},\label{ran_bnd_s2_rws}
		\end{align}
		where 
		\beq\label{def_DeltaS}
		\Delta_{0,\sigma_0}:=  \mu P(x^0,x^*)+\psi(x^0)-\psi^* +\tfrac{\sigma_0^2}{m\mu}.
		\eeq
		% with $\Delta_0$ and $\Delta_{0,\sigma_0}$ defined in \eqref{def_Delta} and \eqref{def_DS}, respectively.
	% \end{itemize} 
\end{theorem}

\vgap
% \todo{Case I}
In view of Theorem~\ref{main_ran_sc}, we can provide bounds on the total number of gradient evaluations performed by RGEM to find a stochastic $\epsilon$-solution of problem \eqnok{cp}, i.e., a point $\bar x \in X$ s.t. $\bbe[\psi(\bar x)-\psi^*] \le \epsilon$. 
% \todo{Case II}
% Furthermore, by 
Theorem~\ref{main_ran_sc} implies the number of gradient evaluations of $f_i$ 
% and the number of communication rounds 
performed by RGEM to find a stochastic $\epsilon$-solution of \eqnok{cp}
% , i.e., a point $\bar x \in X$ s.t. $\bbe[\psi(\bar x)-\psi^*] \le \epsilon$, 
can be bounded by 
% \beq\label{def_bnd_rws1}
% {\cal O}\left\{\sqrt{\tfrac{m\hat LP(x^0,x^*)+m^2(f(x^0)-f^*+\sigma D_X)}{\epsilon}}\right\},
% \eeq
% for smooth convex problems and by 
% \todo{
\beq\label{def_K_rws}
K(\epsilon,C,\sigma_0^2)= 2\left(m+\sqrt{m^2+16mC}\right)\log \tfrac{6\max\{m,C\}\Delta_{0,\sigma_0}}{\epsilon} = {\cal O} \left\{\left(m +\sqrt{\tfrac{m\hat L}{\mu}}\right)\log \tfrac{1}{\epsilon}\right\}.
\eeq
Here $C=\hat L/\mu$.
%Similarly, 
%if we set the distance to the optimal solution, i.e., $\bbe[P(x^k,x^*)]$, as the termination criterion, the number of gradient evaluations of $f_i$ performed by the RGEM method to find a stochastic $\epsilon$-solution and a stochastic $(\epsilon,\lambda)$-solution of \eqnok{cp} can be bounded by $\tilde K(\epsilon,C)$ and $\tilde K(\lambda\epsilon,C)$, respectively, where
%\beq\label{def_tildeK}
%\tilde K(\epsilon,C)= \tfrac{m+\sqrt{m^2+8mC}}{2}\log \left[\tfrac{2(P(x^0,x^*)+(\psi(x^0)-\psi^*)/\mu)}{\epsilon}\right] + m.
%\eeq
%Moreover, 
%comparing the complexity results in \eqref{def_K} and \eqref{def_tildeK} with those of any optimal deterministic first-order methods (FOMs), they differ in a factor of ${\cal O}(\sqrt{mL_f/\hat L})$ 
Therefore, whenever $\sqrt{mC}\log(1/\epsilon)$ is dominating, and $L_f$ and $\hat L$ are in the same order of magnitude, 
RGEM can save up to ${\cal O}(\sqrt{m})$ gradient evaluations of the component function $f_i$ than the optimal deterministic first-order methods.
More specifically, RGEM does not require any exact gradient computation and its communication cost is similar to pure stochastic gradient descent. 
To the best of our knowledge, it is the first time that such an optimal RIG method is presented for solving \eqref{cp} in the literature.
 % even without any exact gradient evaluation of $f$. 
%However, it should be pointed out that $L_f$ can be much smaller than $\hat L$.
It should be pointed out that while the rates of convergence of RGEM obtained in Theorem~\ref{main_ran_sc} is stated in terms of expectation,
we can develop large-deviation results for these rates of convergence using similar techniques in \cite{lan2015optimal} 
for solving strongly convex problems. 

% \todo{
Furthermore, if a one-time exact gradient evaluation is available at the initial point, i.e., $y^{-1}=y^0=(\nabla f_1(x^0),\dots, \nabla f_m(x^0))$, we can drop the assumption in \eqref{def_sigma} and employ a more aggressive stepsize policy with 
\[
	\alpha=1-\tfrac{2}{m+\sqrt{m^2+8m\hat L/\mu}},
\]
Similarly, we can demonstrate that the number of gradient evaluations of $f_i$ performed by RGEM with this initialization method to find a 
stochastic $\epsilon$-solution can be bounded by
% \beq\label{def_K}
\[
\left(m+\sqrt{m^2+8mC}\right)\log\left( \tfrac{6\max\{m,C\}\Delta_{0,0}}{\epsilon} \right) +m= {\cal O} \left\{\left(m +\sqrt{\tfrac{m\hat L}{\mu}}\right)\log \tfrac{1}{\epsilon}\right\}.
\]
% }
% Moreover, since each iteration of RGEM involves only the communication between the server and one agent in the network,
% totally it can save up to ${\cal O}\{\sqrt{m}\}$ rounds of communication than the optimal deterministic first-order methods such as GEM.
% with $C=\hat L/\mu=\max_{i=1,\dots,m}{L_i}/\mu$ for strongly convex problems. Therefore, comparing to the optimal deterministic first-order methods, distributed RGEM can still save up to ${\cal O}(\sqrt{m})$ gradient evaluations of $f_i$ without any exact gradient computation, and also save up to ${\cal O}(\sqrt{m})$ rounds of communication under the distributed setting for strongly convex case.
% Moreover, distributed RGEM also relaxes 
% zero exact gradient computation of $f$ prevents distributed RGEM from synchronous delays. 
% While the above rate of convergence of distributed RGEM is stated in terms of expectation, we can also develop large-deviation results using similar approaches in \cite{lan2015optimal}.

\vgap
\subsection{RGEM for stochastic finite-sum optimization}\label{sec_online}
% \todo{discuss online RGEM}
We discuss in this subsection the stochastic finite-sum optimization and online learning problems, where only noisy gradient information of $f_i$ can be accessed via a stochastic first-order ($\SO$) oracle. In particular, for any given point $\underline{x}_i^t\in X$, the $\SO$ oracle outputs a vector $G_i(\underline{x}_i^t,\xi_i^t)$ s.t.
\begin{align}
\bbe_{\xi}[G_i(\underline{x}_i^t,\xi_i^t)] = \nabla f_i(\underline x_i^t)&,\ i=1,\ldots,m,\label{assmp_unbias}\\
\bbe_{\xi}[\|G_i(\underline{x}_i^t,\xi_i^t)-\nabla f_i(\underline{x}_i^t)\|_*^2]\le \sigma^2&, \ i=1,\ldots,m.\label{assump_var}
\end{align}
We also assume that throughout this subsection that the $\|\cdot\|$ is associated with the inner product $\la \cdot, \cdot\ra$.

\vgap
As shown in Algorithm~\ref{alg_rpaged_sto}, the 
% \beq\label{sp}
%  \psi^*:=\min_{x \in X} \left\{\psi(x):=\tfrac{1}{m}\tsum_{i=1}^{m}\bbe_{\xi_i}[F_i(x,\xi_i)]+\mu w(x)\right\}.
%  \eeq
RGEM for stochastic finite-sum optimization is naturally obtained by replacing the gradient evaluation of $f_i$ in Algorithm~\ref{alg_rpaged} (see \eqref{def_yt1}) with a stochastic gradient estimator of $f_i$ given in \eqref{def_yt1_sto}. In particular, at each iteration, we collect $B_t$ number of stochastic gradients of only one randomly selected component $f_i$ and take their average as the stochastic estimator of $\nabla f_i$. 
% We formally described RGEM for solving stochastic finite-sum problems as follows.
Moreover,
% Since Algorithm~\ref{alg_rpaged_sto} is designed to solve stochastic optimization problems \eqref{sp}, 
it needs to be mentioned that the way RGEM initializes gradients, i.e, $y^{-1}=y^0=\0b$, is very important for stochastic optimization, since it is usually impossible to compute exact gradient for expectation functions even at the initial point.

\begin{algorithm}[H]
	\caption{RGEM for stochastic finite-sum optimization}
	\label{alg_rpaged_sto}
	\begin{algorithmic}
	\State
	This algorithm is the same as Algorithm~\ref{alg_rpaged} except 
	% that Type II initialization is used and 
	that \eqref{def_yt1} is replaced by 
		% \State {\bf Input:}
		% Let $x^0\in X$, and the nonnegative parameters $\{\alpha_t\}$, $\{\eta_t\}$, $\{\tau_t\}$, and $\{B_t\}$ be given.
		% %		\State {\bf Warm start stage:} $\bar x= Algorithm\ref{alg_spdg}(x_0,d,\{\tilde \tau_t\},\{\tilde \eta_t\},\{\tilde \alpha_t\})$.
		% \State
		% Set $\underline{x}_{i}^0=x^0, \ i=1,\ldots,m$, and $y^{-1}=y^0=\0b$. 
		% %and $g^0=\tfrac{1}{m}\tsum_{i=1}^my_i^0$.
		% \State
		% {\bf for} $t=1,\dots,k$ {\bf do}
		
		% Choose $i_t$ according to $\Prob\{i_t=i\}=\tfrac{1}{m}, i=1,\dots,m$.
		
		% Update $z^t=(x^t,y^t)$ according to \eqref{def_tildeyt} - \eqref{def_underlx} and
		% \begin{align}
		% \tilde{y}^t&=y^{t-1} + \alpha_t(y^{t-1}-y^{t-2}),\label{def_tildeyt}\\
		% %		\tilde{g}^t&=g^{t-1}+\tfrac{\alpha_t}{m}(y_{i_{t-1}}^{t-1}-y_{i_{t-1}}^{t-2}),\label{def_tildeg}\\
		% x^t&=\mathcal{M}_{X}(\tfrac{1}{m}\tsum_{i=1}^m\tilde y^t_i,x^{t-1},\eta_t),\label{def_xt1}\\
		% \underline{x}_{i}^t&=
		% \begin{cases}
		% (1+\tau_t)^{-1}(x^t+\tau_t\underline{x}^{t-1}_{i}), \ \ i=i_t,\\
		% \underline{x}_{i}^{t-1}, \ \ i\ne i_t.
		% \end{cases}\label{def_underlx}\\
		\State
		\beq\label{def_yt1_sto}
		y_i^t=
		\begin{cases}
		\tfrac{1}{B_t}\tsum_{j=1}^{B_t}G_i(\underline{x}_{i}^t,\xi_{i,j}^t), \ \ i=i_t,\\
		y_i^{t-1}, \ \ i\ne i_t.
		\end{cases}
		\eeq
		\State
		Here, $G_i(\underline{x}_{i}^t,\xi_{i,j}^t)$, $j = 1,\ldots,B_t$, are stochastic gradients of $f_i$ computed by the $\SO$ oracle at $\underline{x}_{i}^t$.
		%		g^t&=g^{t-1}+\tfrac{1}{m}(y_{i_t}^{t}-y_{i_t}^{t-1}), \label{def_gt_r}
		% \end{align}  
		% \State {\bf end for}
		% \State {\bf Output:} 
		% $\underline x^k:=(\tsum_{t=1}^k\theta_t)^{-1}\tsum_{t=1}^k\theta_t x^t$, for some $\theta_t>0$ (cf. \eqref{def_ergodic_m}).
	\end{algorithmic}
\end{algorithm}

%  and hence we require the assumption that there exists $\sigma_0\ge 0$ satisfying \eqref{def_sigma}. 
% Note that whenever Case I initialization is available, the above assumption is not necessarily for the convergence of Algorithm~\ref{alg_rpaged_sto}, and the convergence results established later in Theorem~\ref{main_ran_sto} still hold with $\sigma_0=0$.

Under the standard assumptions in \eqref{assmp_unbias} and \eqref{assump_var} for stochastic optimization, and with proper choices of algorithmic parameters, Theorem~\ref{main_ran_sto} shows that RGEM can achieve the optimal ${\cal O}\{\sigma^2/\mu^2\epsilon\}$ rate of convergence (up to a certain logarithmic factor) for solving strongly convex problems given in the form of \eqref{sp} in terms of the number of stochastic gradients of $f_i$. 
% To the best of our knowledge, it is the first time that such a convergence result is established for RIG algorithms to solve distributed stochastic finite-sum problems. 
The proof of the this result can be found in Subsection~\ref{sec_sto}.

\vgap
\begin{theorem}\label{main_ran_sto}
	Let $x^*$ be an optimal solution of \eqnok{sp}, $x^k$ and $\underline x^k$ be generated by Algorithm~\ref{alg_rpaged_sto}, and $\hat L=\max_{i=1,\dots,m}{L_i}$. 
	Suppose that $\sigma_0$ and $\sigma$ are defined in \eqref{def_sigma} and \eqref{assump_var}, respectively.
	Given the iteration limit $k$,  let $\{\tau_t\}$, $\{\eta_t\}$ and $\{\alpha_t\}$ be set to \eqref{constant_step_s_ran} with $\alpha$ being set as \eqref{def_nalpha}, and we also set 
	\beq\label{def_bt}
	B_t = \lceil k(1-\alpha)^2\alpha^{-t}\rceil, \ t =1,\ldots,k,
	\eeq
	then
	\begin{align}
		\bbe[P(x^k,x^*)]&\le \tfrac{2\alpha^k\Delta_{0,\sigma_0,\sigma}}{\mu} ,\label{ran_bnd_s1_sto}\\
		\bbe[\psi(\underline{x}^k)-\psi(x^*)]&\le 6\max\left\{m,\tfrac{\hat L}{\mu}\right\} \Delta_{0,\sigma_0,\sigma}\alpha^{k/2},	\label{ran_bnd_s2_sto}
		\end{align}
		where the expectation is taken w.r.t. $\{i_t\}$ and $\{\xi_{i}^t\}$ and 
		\beq\label{def_DSS}
		\Delta_{0,\sigma_0,\sigma}:= \mu P(x^0,x^*)+\psi(x^0)-\psi(x^*) + \tfrac{\sigma_0^2/m+5\sigma^2}{\mu}.
		\eeq
		% with $\Delta_{0,\sigma_0.\sigma}$ defined in \eqref{def_DSS}.
\end{theorem}

\vgap
In view of \eqref{ran_bnd_s2_sto}, the number of iterations performed by RGEM to find a stochastic $\epsilon$-solution of \eqnok{sp}, 
% i.e., a point $\bar x \in X$ s.t. $\bbe[\psi(\bar x)-\psi^*] \le \epsilon$, 
can be bounded by  
% \beq\label{def_bnd_rws1}
% {\cal O}\left\{\sqrt{\tfrac{m\hat LP(x^0,x^*)+m^2(f(x^0)-f^*+\sigma D_X)}{\epsilon}}\right\},
% \eeq
% for smooth convex problems and by 
\beq\label{def_hK_sto}
\hat K(\epsilon,C,\sigma_0^2,\sigma^2) := 2\left(m+\sqrt{m^2+16mC}\right)\log \tfrac{6\max\{m,C\} \Delta_{0,\sigma_0,\sigma}}{\epsilon}.
% =\log_{1/\alpha} \left[\tfrac{4\max\{m,C\}\tilde \Delta_{0,\sigma_0,\sigma}}{\epsilon}\log_{1/\alpha}\tfrac{2\max\{m,C\}\tilde \Delta_{0,\sigma_0,\sigma}}{\epsilon}\right],
\eeq
% with $C=\hat L/\mu$ for strongly convex problems. 
Furthermore, in view of \eqref{ran_bnd_s1_sto} this iteration complexity bound can be improved to 
\beq\label{def_bK_sto}
\bar K(\epsilon,\alpha,\sigma_0^2,\sigma^2):= \log_{1/\alpha}\tfrac{2\tilde \Delta_{0,\sigma_0,\sigma}}{\mu\epsilon},
\eeq
in terms of finding a point $\bar x \in X$ s.t. $\bbe[P(\bar x,x^*)] \le \epsilon$. Therefore, the corresponding number of stochastic gradient evaluations performed by RGEM for solving problem \eqref{sp} can be bounded by 
% \beq\label{complexity_sfo_psi}
% \tsum_{t=1}^k B_t \le  k\tsum_{t=1}^k(1-\alpha)\alpha^{-t} + k
% \le k(1/\alpha)^k = {\cal O}\left\{\tfrac{m\hat LC_0}{\mu\epsilon}\left(\log\tfrac{C_0}{\epsilon}\right)^2\right\},
% \eeq
% and 
\beq\label{complexity_sfo_p}
\tsum_{t=1}^k B_t \le  k\tsum_{t=1}^k(1-\alpha)^2\alpha^{-t} + k =
{\cal O}\left\{\left(\tfrac{\Delta_{0,\sigma_0,\sigma}}{\mu\epsilon} + m + \sqrt{mC}\right)\log\tfrac{\Delta_{0,\sigma_0,\sigma}}{\mu\epsilon}\right\}, 
\eeq
which together with \eqref{def_DSS} imply that the total number of required stochastic gradients or samples of the random variables $\xi_i, \ i=1,\ldots,m$, can be bounded by
% Therefore, the number of gradient evaluations of $f_i$ and the number of communication rounds performed by distributed RGEM to find a point $\bar x \in X$ s.t. $\bbe[P(\bar x,x^*)] \le \epsilon$ respectively, can be bounded by 
\[
	{\cal \tilde O}\left\{\tfrac{\sigma_0^2/m+\sigma^2}{\mu^2\epsilon}+\tfrac{\mu P(x^0,x^*) +\psi(x^0)-\psi^*}{\mu\epsilon}+ m + \sqrt{\tfrac{m\hat L}{\mu}}\right\}.
\]
% rate of convergence for RGEM to solve the class of stochastic finite-sum optimization problems defined in \eqref{sp}.  
% and
% \[
% {\cal O} \left\{\left(m +\sqrt{\tfrac{m\hat L}{\mu}}\right)\log \tfrac{1}{\epsilon}\right\}.
% \]
Observe that this bound does not depend on the number of terms $m$ for small enough $\epsilon$.
To the best of our knowledge, it is the first time that such a convergence result is established for RIG algorithms to solve distributed stochastic finite-sum problems. 
This complexity bound in fact is in the same order of magnitude (up to a logarithmic factor) as the complexity bound achieved by the optimal accelerated stochastic approximation methods \cite{GhaLan12-2a,GhaLan13-1,Lan10-3}, which uniformly sample all the random variables $\xi_i, \ i=1,\ldots,m$.
% Moreover, if Case I initialization is employed by RGEM, 
% Moreover, the communication (cf. iteration) complexity is not improvable, since the second bound matches the lower complexity bound \eqref{opt_complexity} obtained in \cite{lan2015optimal}.
% To the best of the authors' knowledge, it is the first time that a RIG type of method for distributed stochastic problems achieves the above complexity bound on the number of calls to the $\SO$ oracle.
% , and the linear rate of convergence in terms of the communication complexity.
However, this latter approach will thus involve much higher communication costs in the distributed setting (see Subsection~\ref{sec_distributed} for more discussions). 

\vgap
\subsection{RGEM for distributed optimization and machine learning}\label{sec_distributed}
This subsection is devoted to RGEMs (see  Algorithm~\ref{alg_rpaged} and Algorithm~\ref{alg_rpaged_sto}) from two different perspectives, i.e., the server and the activated agent under a distributed setting. We also discuss the communication costs incurred by RGEM under this setting.

Both the server and agents in the distributed network start with the same global initial point $x^0$, i.e., 
$
\underline x_i^0 = x^0, \ i = 1, \ldots, m,
$
and the server also sets $\Delta y = \0b$ and $g^0= \0b$. 
% \[
% g^0 = \begin{cases}
% \tfrac{1}{m}\tsum_{i=1}^m f_i(x^0), \ &\text{for Type I initialization}, \\ 
% \0b, \ &\text{for Type II initialization}.
% \end{cases}
% \]
During the process of RGEM, the server updates iterate $x^t$ and calculates the output solution 
$\underline x^k$ (cf. \eqref{def_ergodic_m})
which is given by $\rm{sum}x/\rm{sum}\theta$. Each agent only stores its local variable $\underline{x}_i^t$ and updates it according to the information received from the server (i.e., $x^{t}$) when activated.
The activated agent also needs to upload the changes of gradient $\Delta y_i$ to the server. 
Observe that since $\Delta y$ might be sparse,  uploading it will incur smaller amount of communication costs than uploading the new gradient $y_i^t$.
% Since the upload speeds are typically much slower than download speeds, RGEM only requires uploading $\Delta y$ to the server, which might be sparse, to save communication costs. 
Note that line~\ref{line:history_y} of RGEM from the $i_t$-th agent's perspective is optional if the agent saves historic gradient information from the last update. 
% We formally describe RGEM for the distributed setting as follows.

\newpage
% The learning process based on distributed RGEM is presented below.
\begin{multicols}{2}
\begin{RGEM}[H]
\caption{The server's perspective}
% \label{RGEM_cloud}
\begin{algorithmic}[1]
% \State
\While{$t\le k$}
% \State  
% Send signal to $i_t$ - device when $i_t$ is chosen uniformly from $\{1,\ldots,m\}$
% \If {$i_t$ - device is responsive} 
% \State
% Sending current learning model to $i_t$ - device
% \If {Receive feedback as $\Delta y$}
% \State 
% $ \tfrac{1}{m}\tsum_{i=1}^m\tilde y^t_i \gets g^{t-1} + \tfrac{\alpha}{m}\Delta y 
% $
\State \label{line:prox-map}$x^t \gets\mathcal{M}_{X}(g^{t-1} + \tfrac{\alpha_t}{m}\Delta y,x^{t-1},\eta_t)$ 
% \Comment{Defined in \eqref{prox_mapping}}

% \State
% $
% \underline x^t \gets \underline x^{t-1} + \theta_t x^{t-1} $
% \State $ \tilde y^t \gets g^{t-1} + \tfrac{1+\alpha}{m}\Delta y_i, 
% $
\State
$
\rm{sum}x \gets \rm{sum}x+ \theta_t x^{t} $
\State
$
\rm{sum}\theta \gets \rm{sum}\theta+ \theta_t$
\State\label{line:call}Send signal to the $i_t$-th agent where $i_t$ is selected uniformly from $\{1,\ldots,m\}$
\If {$i_t$-th agent is responsive} 
\State
Send current iterate $x^t$ to $i_t$-th agent
\If {Receive feedback $\Delta y$}
\State
$
g^t \gets g^{t-1} + \Delta y
$
\State
$
	t \gets t+1
$
\Else \ \textbf{goto} \emph{Line \ref{line:call}}
\EndIf
\Else \ \textbf{goto} \emph{Line \ref{line:call}} 
\EndIf
\EndWhile
\end{algorithmic}
\end{RGEM}
% \vfill\null
% \columnbreak
\begin{RGEM}[H]
\caption{The activated $i_t$-th agent's perspective}
% \label{RGEM_device}
\begin{algorithmic}[1]
\State
Download the current iterate $x^{t}$ from the server 
% \State
% $x^t \gets\mathcal{M}_{X}(\tfrac{1}{m}\tsum_{i=1}^m\tilde y^t_i,x^{t-1},\eta_t)$
\If {$t = 1$}
\State
$y_i^{t-1} \gets \0b$
\Else 
\State\label{line:history_y} 
$y_i^{t-1} \gets \nabla f_i(\underline x_i^{t-1})$ \Comment{Optional}
\EndIf
\State
$\underline{x}_{i}^t \gets (1+\tau_t)^{-1}(x^t+\tau_t\underline{x}^{t-1}_{i})$
\State
$y_i^t\gets \nabla f_i(\underline{x}_{i}^t)$
\State
Upload the local changes to the server, i.e., $\Delta y_i = y_i^{t} - y_i^{t-1}$
\end{algorithmic}
\end{RGEM}
\end{multicols}
We now add some remarks about the potential benefits of RGEM for distributed optimization and machine learning. 
Firstly, since RGEM does not require any exact gradient evaluation of $f$, it does not need to wait for the responses from all agents in order to compute an exact gradient. Each iteration of RGEM only involves communication between the server and the activated $i_t$-th agent. In fact, RGEM will move to the next iteration in case no response is received from the $i_t$-th agent. This scheme works under the assumption that the probability for any agent being responsive or available at a certain point of time is equal.
%  can be naturally and efficiently applied to solve distributed optimization problems. In fact,
% for simplicity, suppose that the probability for any agent in the distributed network being responsive or available at certain point of time is equal, then equivalently it can be viewed as the server uniformly pick one agent to update iteratively. 
% Secondly, it should be pointed out that each iteration of RGEM only involves communications between the server and the activated agent (i.e., $i_t$-th agent). 
However, all other optimal RIG algorithms, except RPDG \cite{lan2015optimal}, need the exact gradient information from all network agents once in a while, which incurs high communication costs and synchronous delays as long as one agent is not responsive.
Even RPDG requires a full round of communications and synchronization at the initial point.
% Thirdly, RGEM relaxes the assumption of having Lipschitz gradients over $\bbr^n$ required by RPDG to having Lipschitz gradients over the feasible set $X$, i.e., \eqref{def_smoothness}.

% \todo{discussion about communication costs}
Secondly, since each iteration of RGEM involves only constant number of communication rounds between the server and one selected agent, the communication complexity for RGEM under distributed setting can be bounded by
\[
{\cal O} \left\{\left(m +\sqrt{\tfrac{m\hat L}{\mu}}\right)\log \tfrac{1}{\epsilon}\right\}.
\]
Therefore, it can save up to ${\cal O}\{\sqrt{m}\}$ rounds of communication than the optimal deterministic first-order methods.

\vgap
For solving distributed stochastic finite-sum optimization problems \eqref{sp}, RGEM from the $i_t$-th agent's perspective will be slightly modified as follows. 
\begin{RGEM}[H]
\caption{The activated $i_t$-th agent's perspective for solving \eqref{sp}}
\label{RGEM_device_online}
\begin{algorithmic}[1]
\State
Download the current iterate $x^{t}$ from the server
% \State
% $x^t \gets\mathcal{M}_{X}(\tfrac{1}{m}\tsum_{i=1}^m\tilde y^t_i,x^{t-1},\eta_t)$
\If {$t = 1$} 
% \State\label{line:history_y} 
% $y_i^{t-1} \gets \tfrac{1}{B_t}\tsum_{j=1}^{B_t}G_i(\underline{x}_{i}^{t-1},\xi_{i,j}^t)$ \Comment{Optional}
% \Else 
\State
$y_i^{t-1} \gets \0b$  \Comment{Assuming RGEM saves $y_i^{t-1}$ for $t\ge 2$ at the latest update}
\EndIf
\State
$\underline{x}_{i}^t \gets (1+\tau_t)^{-1}(x^t+\tau_t\underline{x}^{t-1}_{i})$
\State
$y_i^t\gets \tfrac{1}{B_t}\tsum_{j=1}^{B_t}G_i(\underline{x}_{i}^t,\xi_{i,j}^t)$ \Comment{$B_t$ is the batch size, and $G_i$'s are the stochastic gradients given by $\SO$} 
\State
Upload the local changes to the server, i.e., $\Delta y_i = y_i^{t} - y_i^{t-1}$
\end{algorithmic}
\end{RGEM} 

% \todo{
Similar to the case for the deterministic finite-sum optimization, the total number of communication rounds performed by the above RGEM  
% to find a point $\bar x \in X$ s.t. $\bbe[P(\bar x,x^*)] \le \epsilon$ 
can be bounded by 
% \[
% 	{\cal \tilde O}\left\{\tfrac{\sigma_0^2/m+\sigma^2}{\mu^2\epsilon}+\tfrac{\mu P(x^0,x^*) +\psi(x^0)-\psi^* }{\mu\epsilon} + m+ \sqrt{\tfrac{m\hat L}{\mu}}\right\},
% \]
% and
\[
{\cal O} \left\{\left(m +\sqrt{\tfrac{m\hat L}{\mu}}\right)\log \tfrac{1}{\epsilon}\right\},
\]
for solving \eqref{sp}.
Each round of communication only involves the server and a randomly selected agent. 
This communication complexity seems to be optimal, since it matches the lower complexity bound \eqref{RIG_lb} established in \cite{lan2015optimal}.
Moreover, the sampling complexity, i.e., the total number of samples to be collected by all the agents, is also nearly optimal and comparable to the case when all these samples are collected in a centralized location and processed by an optimal stochastic approximation method.
On the other hand, if one applies an existing optimal stochastic approximation method to solve the distributed stochastic optimization problem, the communication complexity will be as high as ${\cal O}(1/\sqrt\epsilon)$, which is much worse than RGEM.

\vgap
% Before we present the proofs of our main results, we will introduce the deterministic version of RGEM, namely the gradient extrapolation method (GEM), in Section~\ref{sec_deter} to provide some insights about the design of the algorithmic scheme and facilitate the understandings of the convergence analysis of RGEM.

%%%%% Lan starts to revise the paper here %%%%%%%%%%%%%%%%
\setcounter{equation}{0}
\section{Gradient extrapolation method: dual of Nesterov's acceleration}\label{sec_deter}
Our goal in this section is to introduce a new algorithmic framework, referred to as the gradient extrapolation method (GEM),
for solving the convex optimization problem given by
\beq\label{cp1}
\psi^*:=\min_{x \in X} \left\{\psi(x):= f(x)+\mu w(x)\right\}.
\eeq
We show that GEM can be viewed as a dual of Nesterov's accelerated
gradient method although these two algorithms appear to be quite different.
Moreover, GEM possess some nice properties which enable us to develop and analyze
the random gradient extrapolation method for distributed and stochastic optimization.
% \section{Preliminaries}

\vgap
% \subsection{Preliminaries}
\subsection{Generalized Bregman distance}\label{sec_pre}
%Our goal in this subsection is to define the concept of generalized Bregman distance(see \cite{Breg67},\cite{AuTe06-1},\cite{BBC03-1},\cite{Kiw97-1} and references therein), which is defined as a prox-function with possibly non-differentiable distance generating function.
%While the standard technique for using Bregman distance (or prox-function) to solve strongly convex problems is to assume that it growths quadratically, we utilize the strong convexity of the objective function to avoid this assumption in the convergence analysis.
In this subsection, we provide a brief introduction to the generalized Bregman distance defined in \eqref{primal_prox} and 
some properties for its associated prox-mapping defined in\eqref{prox_mapping}.

% \vgap
% Observe that the function $w(\cdot)$ in \eqref{cp} is strongly convex with modulus $1$ w.r.t. $\|\cdot\|$. We define a {\sl prox-function} associated with $w$ as
% \beq\label{primal_prox}
% P(x^0,x)\equiv P_{w}(x^0,x):=w(x)-\left[w(x^0)+\langle w'(x^0),x-x^0\rangle\right],
% \eeq
% where $w'(x^0)\in \partial w(x^0)$ is an arbitrary subgradient of $w$ at $x^0$. By the strong convexity of $w$,
% we have 
% \beq\label{P_strong}
% P(x^0,x)\ge \frac{1}{2}\|x-x^0\|^2, \ \ \forall x,x^0\in X.
% \eeq
% It should be pointed out that the prox-function $P(\cdot,\cdot)$ described above is a generalized Bregman distance in the sense that $w$ is not necessarily differentiable. 
% This is different from the standard definition for Bregman distance~\cite{Breg67,AuTe06-1,BBC03-1,Kiw97-1,censor1981iterative}. 
% %\todo{I cited his 1981 paper(please cite to the paper for one Israel professor who contact us)}.
% Throughout this paper, we define the prox-mapping associated with $X$ and $w$, given by
% \beq\label{prox_mapping}
% \mathcal{M}_{X}(g,x^0,\eta):=\argmin_{x\in X}\left\{\langle g,x\rangle+ \mu w(x)+\eta P(x^0,x) \right\},
% \eeq
% which is easily computable for any $x^0\in X, g\in \bbr^{n}, \mu\ge 0, \eta>0$. 
Note that whenever $w$ is non-differentiable, we need to specify a
particular selection of the subgradient $w'$ before performing the prox-mapping.
We assume throughout this paper that such a selection of $w'$ is 
defined recursively as follows. Denote $x^1 \equiv \Pr_X(g, x^0, \eta)$.
By the optimality condition of \eqnok{prox_mapping}, 
we have 
\[
g + (\mu + \eta) w'(x^1) -\eta w'(x^0) \in {\cal N}_X(x^1),
\]
where ${\cal N}_X(x^1):= \{v\in \bbr^n:v^T(x- x^1)\le 0, \forall x\in X\}$ denotes the normal cone of $X$ at $x^1$.
Once such a $w'(x^1)$ satisfying the above relation
is identified, we will use it as a subgradient when defining
$P(x^1, x)$ in the next iteration. 
Note that such a subgradient can be identified as long as $x^1$ is obtained, since it satisfies the optimality condition of \eqnok{prox_mapping}.

\vgap

%subsection{Characteristic of the solutions to the prox-mappings}
%In this subsection, we are devoted to characterize
The following lemma, which generalizes   Lemma 6 of \cite{LaLuMo11-1} and Lemma 2 of \cite{GhaLan12-2a}, characterizes the solutions to \eqnok{prox_mapping}. The proof of this result can be found in Lemma 5 of  \cite{lan2015optimal}.

\begin{lemma} \label{opt_condi}
	Let $U$ be a closed convex set and a point $\tilde u \in U$ be given. 
	Also let $w:U \to \bbr$ be a convex function and
	\[
	W(\tilde u, u) = w(u) - w(\tilde u) - \langle w'(\tilde u), u - \tilde u \rangle
	% \eeq
	\]
	for some $w'(\tilde u) \in \partial w(\tilde u)$.
	Assume that the function $q: U \to \bbr$ satisfies
	% \beq \label{strong_q}
	\[
	q(u_1) - q(u_2) - \langle q'(u_2), u_1 - u_2 \rangle \ge \mu_0 W(u_2, u_1), \ \ \forall u_1, u_2 \in U
	% \eeq
	\]
	for some $\mu_0 \ge 0$.
	Also assume that the scalars $\mu_1$ and $\mu_2$ 
	are chosen such that $\mu_0 + \mu_1 + \mu_2 \ge 0$.
	If
	% \beq \label{tech_lemma_prob}
	\[
	u^* \in \Argmin \{ q(u) + \mu_1 w(u) + \mu_2 W(\tilde u, u) : u \in U\},
	% \eeq
	\]
	then for any $ u \in U$, we have
	\[
	q(u^*) + \mu_1 w(u^*) + \mu_2 W(\tilde u, u^*) + (\mu_0 + \mu_1 + \mu_2) W(u^*,u) 
	\le q(u) + \mu_1 w(u) + \mu_2 W(\tilde u, u).
	\]
\end{lemma}
%
%\begin{proof}
%	Let $\phi(u) := q(u) + \mu_1 w(u) + \mu_2 W(\tilde u, u)$.
%	It can be easily checked that for any $u_1, u_2 \in U$, 
%	\begin{align*}
%	W(\tilde u, u_1) &= W(\tilde u, u_2) + \langle W'(\tilde u, u_2), u_1-u_2 \rangle + W(u_2, u_1), \\
%	w(u_1) &= w(u_2) + \langle w'(u_2), u_1 - u_2 \rangle + W(u_2,u_1).
%	\end{align*}
%	Using these relations and \eqnok{strong_q}, we conclude that
%	\beq \label{tech_tt}
%	\phi(u_1) - \phi(u_2) - \langle \phi'(u_2), u_1 - u_2 \rangle \ge (\mu_0 + \mu_1 + \mu_2) W(u_2, u_1)
%	\eeq
%	for any $u_1, u_2 \in Y$, which together with the fact that $\mu_0+\mu_1 + \mu_2 \ge 0$ then imply
%	that $\phi$ is convex. Since $u^*$ is an optimal solution of \eqnok{tech_lemma_prob},
%	we have $\langle \phi'(u^*), u - u^* \rangle \ge 0$. Combining this inequality with \eqnok{tech_tt},
%	we conclude that 
%	\[
%	\phi(u) - \phi(u^*) \ge (\mu_0 + \mu_1 + \mu_2) W(u^*, u), 
%	\]
%	from which the result immediately follows.
%\end{proof}
% \vgap
% \setcounter{equation}{0}
% \section{Dual of Nesterov's accelerated gradient method}\label{sec_deter}

\vgap
\subsection{The algorithm}\label{sec_deter_alg}
As shown in Algorithm~\ref{alg_paged}, GEM starts 
with a gradient extrapolation step \eqref{def_tgt_i} to compute $\tilde g^t$ from the two previous gradients $g^{t-1}$ and
$g^{t-2}$. Based on $\tilde g^t$, it performs a proximal gradient descent step in \eqref{def_xt_i} and updates
the output solution $\underline x^t$. Finally, the gradient at $\underline x^t$ is computed
for gradient extrapolation in the next iteration. 
This algorithm is a special case of RGEM in Algorithm~\ref{alg_rpaged} (with $m=1$).
\begin{algorithm} [H]
	\caption{An optimal gradient extrapolation method (GEM)}
	\label{alg_paged}
	\begin{algorithmic}
		\State 
		\noindent {\bf Input:}  Let $x^0 \in X$, and the nonnegative parameters $\{\alpha_t\}$, $\{\eta_t\}$, and $\{\tau_t\}$ be given. 
		\State Set $\underline x^0 = x^0$ and $g^{-1}=g^0=\nabla f(x^0)$.
		\State {\bf for} $t = 1, 2, \ldots, k$ {\bf do}
		\begin{align}
		\tilde g^{t} &= \alpha_t (g^{t-1} - g^{t-2}) + g^{t-1}.\label{def_tgt_i}\\
		x^{t} &= \Pr_{X} (\tilde g^t, x^{t-1}, \eta_t). \label{def_xt_i} \\
		\underline x^{t} &= \left( x^t + \tau_t \underline x^{t-1} \right) / (1+\tau_t). \label{def_yt_al1} \\
		g^{t} &= \nabla f(\underline x^{t}). \label{def_yt_al2}
		\end{align}
		\State {\bf end for}
		\State {\bf Output:} $\underline{x}^k$.
	\end{algorithmic}
\end{algorithm} 

We now show that GEM can be viewed as the dual of the well-known Nesterov's accelerated gradient (NAG) method as studied in \cite{lan2015optimal}. 
To see such a relationship, we will first rewrite GEM in a primal-dual form.
Let us consider the dual space ${\cal G}$, where the
gradients of $f$ reside, and equip it with the conjugate norm $\|\cdot\|_*$. 
Let $J_f: {\cal G} \to \bbr$ be the conjugate function of $f$ such that
$
f(x) := \max_{g \in {\cal G}} \{\langle x, g \rangle - J_f(g)\}.
$
We can reformulate the original problem in \eqref{cp1} as the following saddle point problem:
\beq \label{cpsaddle_f}
\psi^* := \min_{x \in X} \left\{ \max_{g \in {\cal G} }  \{\langle x, g \rangle - J_f(g)\}+ \mu \, w(x) \right\}.
\eeq
It is clear that $J_f$ is strongly convex
with modulus $1/ L_f$ w.r.t. $\|\cdot\|_*$ (See Chapter E in \cite{hiriart2012fundamentals} for details). Therefore, we 
can define its associated dual generalized Bregman distance and dual prox-mappings as
\begin{align} 
D_f(g^0, g) &:= J_f(g) - [J_f(g^0) + \langle J_f'(g^0), g - g^0 \rangle], \label{def_Df}\\
\Pr_{{\cal G}} (-\tilde x, g^0, \tau) &:= \arg\min\limits_{g \in {\cal G}} \left\{ 
\langle - \tilde x, g \rangle  + J_f(g)+ \tau D_f(g^0, g)\right\},  \label{d_prox_mapping_f}
\end{align}
for any $g^0, g \in {\cal G}$. 
The following result, whose proof is given in Lemma 1 of \cite{lan2015optimal}, shows that
the computation of the dual prox-mapping associated with $D_f$ is equivalent
to the computation of $\nabla f$.

\vgap

\begin{lemma} \label{lemma_dual_move}
Let $\tilde x \in X$ and $g^0 \in {\cal G}$ be given and $D_f(g^0, g)$
be defined in \eqnok{def_Df}. For any $\tau > 0$, let us
denote $z = [\tilde x + \tau J_f'(g^0)] / (1 + \tau)$.
Then we have $\nabla f(z) = \Pr_{{\cal G}} ( -\tilde x, g^0, \tau)$.
\end{lemma}

\vgap

Using this result, we can see that the GEM iteration 
can be written a primal-dual form. Given $(x^0,g^{-1}, g^0)\in X\times \mathcal{G}\times \mathcal{G}$, it updates $(x^t,g^t)$ by
	\begin{align}
\tilde g^{t}&=\alpha_t (g^{t-1}-g^{t-2})+ g^{t-1}, \label{def_alt_1}\\
x^{t} &= \Pr_{X} (\tilde g^t, x^{t-1}, \eta_t), \label{def_alt_2}\\
g^{t} &= \Pr_{{\cal G}} (- x^t, g^{t-1}, \tau_t), \label{def_alt_3}
\end{align}
with a specific selection of $J_f'(g^{t-1}) = \underline x^{t-1}$ in $D_f(g^{t-1}, g)$.
Indeed, by denoting $\underline x^0 = x^0$,
we can easily see from $g^0 = \nabla f(\underline x^0)$ that $\underline x^0 \in \partial J_f(g^0)$.
Now assume that $g^{t-1} = \nabla f(\underline x^{t-1})$ and hence that $\underline x^{t-1} \in \partial J_f(g^{t-1})$.
By the definition of $g^t$ in \eqnok{def_alt_3} and Lemma~\ref{lemma_dual_move}, we conclude that
$g^t = \nabla f(\underline x^t)$ with $\underline x^t = (x^t + \tau_t \underline x^{t-1}) / (1 + \tau_t)$,
which are exactly the definitions given in \eqnok{def_yt_al1} and \eqnok{def_yt_al2}.

Recall that in a simple version of the NAG method (e.g., \cite{Nest04,tseng08-1,Lan10-3,GhaLan12-2a,GhaLan13-1,GhaLan13-2}),
given $(x^{t-1}, \bar x^{t-1}) \in X \times X$,
it updates $(x^t, \bar x^t)$ by
\begin{align}
\underline x^t &= (1-\lambda_t) \bar x^{t-1} + \lambda_t x^{t-1},  \label{AG1} \\
g^t& = \nabla f(\underline x^t),  \label{AG2} \\
x^t &= \Pr_X(g^t, x^{t-1}, \eta_t),  \label{AG3} \\
\bar x^t &= (1-\lambda_t) \bar x^{t-1} + \lambda_t x^t, \label{AG4}
\end{align}
for some $\lambda_t \in [0,1]$. Moreover, 
we have shown in \cite{lan2015optimal} that \eqnok{AG1}-\eqnok{AG4} can be viewed as a specific instantiation of
the following primal-dual updates:
\begin{align}
\tilde x^{t}&=\alpha_t (x^{t-1}-x^{t-2})+ x^{t-1}, \label{def_txt}\\
g^{t} &= \Pr_{{\cal G}} (-\tilde x^t, g^{t-1}, \tau_t), \label{def_yt}\\ 
x^{t} &= \Pr_{X} (g^t, x^{t-1}, \eta_t). \label{def_xt}
 \end{align}
 
Comparing \eqnok{def_alt_1}-\eqnok{def_alt_3} with \eqnok{def_txt}-\eqnok{def_xt},
we can clearly see that GEM is a dual version of NAG, obtained by switching the primal and dual 
variables in each equation of \eqnok{def_txt}-\eqnok{def_xt}. The major difference exists in that the extrapolation step in GEM is performed in
the dual space while the one in NAG is performed in the primal space.
In fact, extrapolation in the dual space will help us
to greatly simplify and further enhance the randomized incremental gradient methods developed in \cite{lan2015optimal} based on NAG.
Another interesting fact is that in GEM, the gradients are computed for the output solutions $\{\underline x^t\}$.
On the other hand, the output solutions in the NAG method are given by $\{\bar x^t\}$ while
the gradients are computed for the extrapolation sequence $\{\underline x^t\}$.

\subsection{Convergence of GEM}

Our goal in this subsection is to establish the convergence properties of the GEM method for solving \eqref{cp1}.
Observe that our analysis is carried out completely in the primal space and does not rely on the primal-dual interpretation 
described in the previous section. This type of analysis technique appears to be new for solving problem \eqref{cp1} in the literature as it also differs significantly from that of NAG.

We first establish some general convergence properties for GEM for both smooth convex ($\mu=0$) and strongly convex cases ($\mu>0$). 

\begin{theorem}\label{main_deter_pro}
Suppose that $\{\eta_t\}$, $\{\tau_t\}$, and $\{\alpha_t\}$ in GEM satisfy
	\begin{align}
	\theta_{t-1}&=\alpha_t\theta_t, \ \ t=2,\dots,k,\label{alpha_theta_d}\\
	\theta_t\eta_t&\le \theta_{t-1}(\mu+\eta_{t-1}), \ \ t=2,\dots,k,\label{eta_theta_d}\\
	\theta_t\tau_t&=\theta_{t-1}(1+\tau_{t-1}), \ t=2,\dots,k, \label{tau_theta_d}\\
	\alpha_tL_f&\le \tau_{t-1}\eta_t, \ \ t=2,\dots,k, \label{alpha_tau_eta_d}\\
	2L_f&\le \tau_k(\mu+\eta_k), \label{L_tau_eta_d}
	\end{align}
	for some $\theta_t\ge 0$, $t=1,\dots,k$. Then, for any $k\ge 1$ and any given $x\in X$, we have
	\begin{align}\label{main_bnd_d}
	\theta_k(1+\tau_k)[\psi(\underline x^k)-\psi(x)]+\tfrac{\theta_k(\mu+\eta_k)}{2}P(x^k,x)
	\le \theta_1\tau_1[\psi(x^0)-\psi(x)]+\theta_1\eta_1P(x^0,x).
	\end{align} 
\end{theorem}

\begin{proof}
	Applying Lemma~\ref{opt_condi} to \eqref{def_xt_i}, we obtain
	\beq\label{primal_opt}
	\langle x^t-x, \alpha_t(g^{t-1}-g^{t-2})+g^{t-1}\rangle+\mu w(x^t)-\mu w(x) \le \eta_tP(x^{t-1},x)-(\mu+\eta_t)P(x^t,x)-\eta_tP(x^{t-1},x^t).
	\eeq
	Moreover, using the definition of $\psi$, the convexity of $f$, and the fact that $g^t=\nabla f(\underline x^t)$, we have
	\begin{align*}
	(1+\tau_t)f(\underline x^t)+\mu w(x^t)-\psi(x)
	&\le (1+\tau_t)f(\underline x^t)+\mu w(x^t)-\mu w(x)-[f(\underline x^t)+\la g^t,x-\underline x^t\ra]\\
	& = \tau_t[f(\underline x^t)-\la g^t, \underline x^{t}-\underline x^{t-1}\ra] - \la g^t, x-x^t\ra +\mu w(x^t)-\mu w(x)\\
	& \le -\tfrac{\tau_t}{2L_f}\|g^t-g^{t-1}\|_*^2 +\tau_tf(\underline x^{t-1})- \la g^t, x-x^t\ra +\mu w(x^t)-\mu w(x)\\
	&\le -\tfrac{\tau_t}{2L_f}\|g^t-g^{t-1}\|_*^2 +\tau_tf(\underline x^{t-1}) +\la x^t-x,g^t-g^{t-1}-\alpha_t(g^{t-1}-g^{t-2})\ra \\
	&~~~ +\eta_tP(x^{t-1},x)-(\mu+\eta_t)P(x^t,x)-\eta_tP(x^{t-1},x^t),
	\end{align*}
	where the first equality follows from the definition of $\underline x^t$ in \eqref{def_yt_al1}, the second inequality follows from the smoothness of $f$ (see Theorem 2.1.5 in \cite{Nest04}), and the last inequality follows from \eqref{primal_opt}.
	Multiplying both sides of the above inequality by $\theta_t$, and summing up the resulting inequalities from $t=1$ to $k$, we obtain
	\begin{align}\label{rel_theta}
	&\tsum_{t=1}^k\theta_t(1+\tau_t)f(\underline x^t)+\tsum_{t=1}^k\theta_t[\mu w(x^t)-\psi(x)] 
	\le -\tsum_{t=1}^k\tfrac{\theta_t\tau_t}{2L_f}\|g^t-g^{t-1}\|_*^2+\tsum_{t=1}^k\theta_t\tau_tf(\underline x^{t-1})\nn\\
	&~~~+\tsum_{t=1}^k\theta_t\la x^t-x, g^t-g^{t-1}-\alpha_t(g^{t-1}-g^{t-2})\ra \nn\\
	&~~~+ \tsum_{t=1}^k \theta_t[\eta_tP(x^{t-1},x)-(\mu+\eta_t)P(x^t,x)-\eta_tP(x^{t-1},x^t)].
	\end{align}
	Now by \eqref{alpha_theta_d} and the fact that $g^{-1}=g^0$, we have
	\begin{align*}
	&\tsum_{t=1}^k\theta_t\la x^t-x, g^t-g^{t-1}-\alpha_t(g^{t-1}-g^{t-2})\ra \\
	&= \tsum_{t=1}^k\theta_t[\la x^t-x, g^t-g^{t-1}\ra - \alpha_t\la x^{t-1}-x, g^{t-1}-g^{t-2}\ra]
	      -\tsum_{t=2}^k\theta_t\alpha_t \la x^t-x^{t-1}, g^{t-1}-g^{t-2}\ra\\
	&= \theta_k\la x^k-x, g^k-g^{k-1}\ra-\tsum_{t=2}^k\theta_t\alpha_t \la x^t-x^{t-1}, g^{t-1}-g^{t-2}\ra.
	\end{align*}
	Moreover, in view of \eqref{eta_theta_d}, \eqref{tau_theta_d} and the definition of $\underline x^t$ \eqref{def_yt_al1}, we obtain
	\begin{align*}
	\tsum_{t=1}^k \theta_t[\eta_tP(x^{t-1},x)-(\mu+\eta_t)P(x^t,x)]
	%&= \theta_1\eta_1P(x^0,x) + \sum_{t=2}^k [\theta_t\eta_t-\theta_{t-1}(\mu+\eta_{t-1})]P(x^{t-1},x) -\theta_k(\mu+\eta_k)P(x^k,x)\\
	& \varstackrel{\eqref{eta_theta_d}}{\le} \theta_1\eta_1P(x^0,x)-\theta_k(\mu+\eta_k)P(x^k,x),\\
	\tsum_{t=1}^k\theta_t[(1+\tau_t)f(\underline x^t) -\tau_tf(\underline x^{t-1})]
	&\varstackrel{\eqref{tau_theta_d}}{=} \theta_k(1+\tau_k)f(\underline x^k) -\theta_1\tau_1f(\underline x^0),\\
	\tsum_{t=1}^k\theta_t&\varstackrel{\eqref{tau_theta_d}}{=} \tsum_{t=2}^k [\theta_{t}\tau_{t}-\theta_{t-1}\tau_{t-1}]+\theta_k = \theta_k(1+\tau_k)-\theta_1\tau_1, \\
	\theta_k(1+\tau_k)\underline x^k
	&\varstackrel{\eqref{def_yt_al1}}{=}
	%\theta_k(x^k+\tau_k\underline x^{k-1})= 
	\theta_k(x^k+\tfrac{\tau_k}{1+\tau_{k-1}}x^{k-1}+\dots+\tprod_{t=2}^{k}\tfrac{\tau_t}{1+\tau_{t-1}}x^1+\tprod_{t=2}^k\tfrac{\tau_t}{1+\tau_{t-1}}\tau_1x^0)\\
	&\varstackrel{\eqref{tau_theta_d}}{=}\tsum_{t=1}^k\theta_tx^t+\theta_1\tau_1x^0.
	\end{align*}
	The last two relations, in view of the convexity of $w(\cdot)$, also imply that
	\begin{align*}
	\theta_k(1+\tau_k)\mu w(\underline x^k) 
	&\le \tsum_{t=1}^k\theta_t\mu w(x^t)+\theta_1\tau_1\mu w(x^0).
	\end{align*}
	Therefore, by \eqref{rel_theta}, the above relations, and the definition of $\psi$, we conclude that 
	\begin{align}\label{rec_deter}
	\theta_k(1+\tau_k)[\psi(\underline x^k)-\psi(x)]
	&\le \tsum_{t=2}^k\left[-\tfrac{\theta_{t-1}\tau_{t-1}}{2L_f}\|g^{t-1}- g^{t-2}\|_*^2-\theta_t\alpha_t\la x^t-x^{t-1}, g^{t-1}-g^{t-2}\ra-\theta_t\eta_tP(x^{t-1},x^{t})\right]\nn\\
	&~~~-\theta_k\left[\tfrac{\tau_k}{2L_f}\|g^{k}-g^{k-1}\|_*^2-\la x^k-x,g^{k}-g^{k-1}\ra+(\mu+\eta_k)P(x^k,x)\right]+\theta_1\eta_1P(x^0,x)\nn\\
	&~~~ + \theta_1\tau_1[\psi(x^0)-\psi(x)] - \theta_1\eta_1P(x^0,x^1).
	\end{align}
By the strong convexity of $P(\cdot,\cdot)$ in \eqref{P_strong},
 the simple relation that $b\langle u,v\rangle -a\|v\|^2/2\le b^2\|u\|^2/(2a), \ \forall a>0$,
and  the conditions in  \eqref{alpha_tau_eta_d} and \eqref{L_tau_eta_d}, we have 
	\begin{align*}
	&-\tsum_{t=2}^k\left[\tfrac{\theta_{t-1}\tau_{t-1}}{2L_f}\|g^{t-1}- g^{t-2}\|_*^2+\theta_t\alpha_t\la x^t-x^{t-1}, g^{t-1}-g^{t-2}\ra+\theta_t\eta_tP(x^{t-1},x^{t})\right]\\
	& \quad \le \tsum_{t=2}^k\tfrac{\theta_t}{2}\left(\tfrac{\alpha_tL_f}{\tau_{t-1}}-\eta_t\right)\|x^{t-1}-x^t\|^2 \le 0\\
	&-\theta_k\left[\tfrac{\tau_k}{2L_f}\|g^{k}-g^{k-1}\|_*^2-\la x^k-x,g^{k}-g^{k-1}\ra+\tfrac{(\mu+\eta_k)}{2}P(x^k,x)\right]\\
	&\quad \le \tfrac{\theta_k}{2}\left(\tfrac{L_f}{\tau_k}-\tfrac{\mu+\eta_k}{2}\right)\|x^k-x\|^2 \le 0.
	\end{align*}
Using the above relations in \eqref{rec_deter}, we obtain \eqref{main_bnd_d}.
\end{proof}

\vgap

We are now ready to establish the optimal convergence behavior of GEM as a consequence of Theorem~\ref{main_deter_pro}.
We first provide a constant step-size policy which guarantees an optimal linear rate of convergence for the strongly convex case ($\mu > 0$).

\begin{corollary}
	Let $x^*$ be an optimal solution of \eqref{cp}, $x^k$ and $\underline x^k$ be defined in \eqref{def_xt_i}
	and \eqref{def_yt_al1}, respectively. 
	Suppose that  $\mu > 0$, and that $\{\tau_t\}$, $\{\eta_t\}$ and $\{\alpha_t\}$ are set to
	\beq \label{constant_step_s}
	\tau_t \equiv \tau = \sqrt{\tfrac{ 2L_f}{\mu}}, \ \ \ \eta_t \equiv \eta = \sqrt{2 L_f\mu}, \ \ \ \mbox{and} \ \ \ \alpha_t \equiv \alpha =
	\tfrac{ \sqrt{2 L_f/\mu}}{1 + \sqrt{2 L_f/\mu}}, \ \forall t = 1, \ldots, k.
	\eeq 
	Then,
	\begin{align}
	P(x^k,x^*) &\le 2\alpha^k[ P(x^0, x^*) +\tfrac{1}{\mu}(\psi(x^0)-\psi^*)], \label{deter_bnd_s1}\\
	\psi(\underline x^k) - \psi^* &\le \alpha^k \left[\mu P(x^0,x^*)+\psi(x^0)-\psi^*\right].  \label{deter_bnd_s2}
	\end{align}
\end{corollary}
\begin{proof}
	Let us set $\theta_t = \alpha^{-t}, \ t=1,\dots,k$. It is easy to check that the selection of $\{\tau_t\}, \{\eta_t\}$ and $\{\alpha_t\}$
	in \eqref{constant_step_s} satisfies 
	conditions \eqref{alpha_theta_d}-\eqref{L_tau_eta_d}. In view of Theorem~\ref{main_deter_pro} and \eqref{constant_step_s}, we have
	\begin{align*}
	\psi(\underline x^k)-\psi(x^*)+\tfrac{\mu+\eta}{2(1+\tau)}P(x^k,x^*)
	&\le \tfrac{\theta_1\tau}{\theta_k(1+\tau)}[\psi(x^0)-\psi(x^*)] + \tfrac{\theta_1\eta}{\theta_k(1+\tau)} P(x^0,x^*)\\
	& = \alpha^k[\psi(x^0)-\psi(x^*)+\mu P(x^0,x^*)].
	\end{align*}
	It also follows from the above relation, the fact $\psi(\underline x^k)-\psi(x^*)\ge 0$, and \eqref{constant_step_s} that
	\[
	P(x^k,x^*)\le \tfrac{2(1+\tau)\alpha^k}{\mu+\eta}[\mu P(x^0,x^*)+\psi(x^0)-\psi(x^*)]= 2\alpha^k[P(x^0,x^*)+\tfrac{1}{\mu}(\psi(x^0)-\psi(x^*))].
	\]
\end{proof}

\vgap

We now provide a stepsize policy which guarantees the optimal rate of
convergence for the smooth case ($\mu = 0$).
Observe that in smooth case we can estimate the solution quality for the sequence $\{\underline x^k\}$ only.

\begin{corollary}\label{main_deter_sm}
	Let $x^*$ be an optimal solution of \eqref{cp}, and $\underline x^k$ be defined in \eqref{def_yt_al1}. Suppose that $\mu=0$, and that $\{\tau_t\}$, $\{\eta_t\}$ and $\{\alpha_t\}$ are set to
	\beq \label{var_step_1}
	\tau_t =\tfrac{t}{2}, \ \ \ \eta_t = \tfrac{4 L_f}{t},  \ \ \ \mbox{and} \ \ \ \alpha_t = \tfrac{t}{t+1}, \ \forall t = 1, \ldots,k.
	\eeq
	Then,
	\beq  \label{deter_bnd_1}
	\psi(\underline x^k) - \psi(x^*) = f(\underline{x}^k)-f(x^*) \le \tfrac{2}{(k+1)(k+2)}[f(x^0)-f(x^*)+8L_f P(x^0,x^*)]. 
	\eeq
\end{corollary}
\begin{proof}
	Let us set $\theta_t = t+1$, $t=1,\ldots,k$. It is easy to check that the parameters in \eqref{var_step_1} satisfy 
	conditions \eqref{alpha_tau_eta_d}-\eqref{L_tau_eta_d}. In view of \eqref{main_bnd_d} and \eqref{var_step_1}, we conclude that
	\[
	\psi(\underline x^k)- \psi(x^*) \le \tfrac{2}{(k+1)(k+2)}[\psi(x^0)-\psi(x^*)+8L_f P(x^0,x^*)].
	\]
\end{proof}

\vgap
In Corollary~\ref{main_deter_sm1}, we improve the above complexity result in terms of the dependence on $f(x^0)-f(x^*)$ by using a different step-size policy and a slightly more involved analysis for the smooth case ($\mu=0$).

\vgap
\begin{corollary} \label{main_deter_sm1}
	Let $x^*$ be an optimal solution of \eqref{cp}, $x^k$ and $\underline x^k$ be defined in \eqref{def_xt_i}
	and \eqref{def_yt_al1}, respectively.
		Suppose that $\mu=0$, and that $\{\tau_t\}$, $\{\eta_t\}$ and $\{\alpha_t\}$ are set to
		\beq \label{var_step}
		\tau_t =\tfrac{t-1}{2}, \ \ \ \eta_t = \tfrac{6 L_f}{t},  \ \ \ \mbox{and} \ \ \ \alpha_t = \tfrac{t-1}{t}, \ \forall t = 1, \ldots,k.
		\eeq
		Then, for any $k\ge 1$,
		\begin{align}
		\psi(\underline x^k) - \psi(x^*) &= f(\underline{x}^k)-f(x^*) \le \tfrac{12L_f}{ k(k+1)} P(x^0, x^*).  \label{deter_bnd}
		% \\
	% P(x^k,x^*)&\le 2P(x^0,x^*). \label{deter_dis_nexp}
	\end{align}
\end{corollary}
\begin{proof}
	If we set $\theta_t = t,\ t=1,\dots,k$. It is easy to check that 
	the parameters in \eqref{var_step} satisfy conditions \eqref{alpha_theta_d}-\eqref{tau_theta_d} and \eqref{L_tau_eta_d}. However,
	condition \eqref{alpha_tau_eta_d} only holds for $t= 3,\ldots,k$, i.e.,
	\beq\label{alpha_tau_eta_dsm}
	\alpha_tL_f\le \tau_{t-1}\eta_t, \ \ t=3,\dots,k.
	\eeq
	In view of \eqref{rec_deter} and the fact that $\tau_1=0$, we have
	\begin{align*}
	\theta_k(1+\tau_k)[\psi(\underline x^k)-\psi(x)]
	&\le -\theta_2[\alpha_2\la x^2-x^1,g^1-g^0\ra + \eta_2P(x^1,x^2)]- \theta_1\eta_1P(x^0,x^1)\nn\\
	&~~~- \tsum_{t=3}^k\left[\tfrac{\theta_{t-1}\tau_{t-1}}{2L_f}\|g^{t-1}- g^{t-2}\|_*^2 +\theta_t\alpha_t\la x^t-x^{t-1}, g^{t-1}-g^{t-2}\ra+\theta_t\eta_tP(x^{t-1},x^{t})\right]\nn\\
	&~~~-\theta_k\left[\tfrac{\tau_k}{2L_f}\|g^{k}-g^{k-1}\|_*^2-\la x^k-x,g^{k}-g^{k-1}\ra+(\mu+\eta_k)P(x^k,x)\right]+\theta_1\eta_1P(x^0,x)\nn\\
	&\le \tfrac{\theta_1\alpha_2}{2\eta_2}\|g^1-g^0\|_*^2 -\tfrac{\theta_1\eta_1}{2}\|x^1-x^0\|^2
	+\tsum_{t=3}^k\tfrac{\theta_t}{2}\left(\tfrac{\alpha_tL_f}{\tau_{t-1}}-\eta_t\right)\|x^{t-1}-x^t\|^2\nn\\
	&\quad +\tfrac{\theta_k}{2}\left(\tfrac{L_f}{\tau_k}-\tfrac{\eta_k}{2}\right)\|x^k-x\|^2+\theta_1\eta_1P(x^0,x)-\tfrac{\theta_k\eta_k}{2}P(x^k,x)\nn\\ 
	&\le \tfrac{\theta_1\alpha_2L_f^2}{2\eta_2}\|\underline x^1-\underline x^0\|^2-\tfrac{\theta_1\eta_1}{2}\|x^1-x^0\|^2+\theta_1\eta_1P(x^0,x)-\tfrac{\theta_k\eta_k}{2}P(x^k,x)\nn\\
	&\le \theta_1 \left(\tfrac{\alpha_2L_f^2}{2\eta_2}-\eta_1\right)\|x^1-x^0\|^2+\theta_1\eta_1P(x^0,x)-\tfrac{\theta_k\eta_k}{2}P(x^k,x),
	\end{align*}
	where the second inequality follows from the simple relation that $b\langle u,v\rangle -a\|v\|^2/2\le b^2\|u\|^2/(2a), \ \forall a>0$ and \eqref{P_strong}, the third inequality follows from \eqref{alpha_tau_eta_dsm}, \eqref{L_tau_eta_d}, the definition of $g^t$ in \eqref{def_yt_al2} and \eqref{smooth_f}, and the last inequality follows from the facts that $\underline x^0=x^0$ and $\underline x^1 = x^1$ (due to $\tau_1=0$). Therefore, by plugging the parameter setting in \eqref{var_step} into the above inequality, we conclude that 
	\begin{align*}
	% 0\le 
	\psi(\underline x^k)-\psi^*=f(\underline x^k)-f(x^*)
	&\le [\theta_k(1+\tau_k)]^{-1}[\theta_1\eta_1P(x^0,x^*)-\tfrac{\theta_k\eta_k}{2}P(x^k,x)]
	\le \tfrac{12L_f}{k(k+1)}P(x^0,x^*).
	\end{align*}
	% The above relation also clearly implies \eqnok{deter_bnd}.
\end{proof}
% \noindent{\bf Remark 1.}
% Observe that the complexity bound established in Corollary~\ref{main_deter_sm} is optimal in terms of the dependence on $\epsilon$, but it also involves the initial functional value gap $f(x^0)-f(x^*)$. We can improve this complexity result
% % and provide a bound on $P(x^k,x^*)$  
% by using a different step-size policy
% \[
% 	\tau_t =\tfrac{t-1}{2}, \ \ \ \eta_t = \tfrac{6 L_f}{t},  \ \ \ \mbox{and} \ \ \ \alpha_t = \tfrac{t-1}{t}, \ \forall t = 1, \ldots,k,
% \]
% and a slightly more involved analysis for the above parameter setting to obtain the
% \[
% 	{\cal O}\left\{\sqrt{\tfrac{L_fP(x^0,x^*)}{\epsilon}}\right\}
% \]
% complexity bound for solving smooth problems ($\mu=0$). Such an analysis can be found in an online and extended version of this paper.

\vgap
In view of the results obtained in the above three corollaries, 
GEM exhibits optimal rates of convergence for both strongly convex and smooth cases.
Different from the classical NAG method, GEM performs extrapolation on the gradients, rather than the iterates.
This fact will help us to develop an enhanced randomized
incremental gradient method than RPDG in \cite{lan2015optimal}, i.e., the Random Gradient Extrapolation Method, 
with a much simpler analysis.% as shown in next section.

\setcounter{equation}{0}
\section{Convergence analysis of RGEM}\label{sec_conv}
Our main goal in this section is to establish the convergence properties of RGEM for solving \eqref{cp} and \eqref{sp}, i.e., the main results stated in Theorem~\ref{main_ran_sc} and~\ref{main_ran_sto}. 
In fact, comparing RGEM in Algorithm~\ref{alg_rpaged} with GEM in Algorithm~\ref{alg_paged}, RGEM is a direct randomization of GEM. Therefore, inheriting from GEM, its convergence analysis is carried out completely in the primal space.
However, the analysis for RGEM is more challenging especially because we need to 1) build up the relationship between $\tfrac{1}{m}\tsum_{i=1}^mf_i(\underline x_i^k)$ and $f(\underline x^k)$, for which we exploit the function $Q$ defined in \eqref{def_Q} as an intermediate tool; 2) bound the error caused by inexact gradients at the initial point and 3) analyze the accumulated error caused by randomization and noisy stochastic gradients.

Before proving Theorem~\ref{main_ran_sc} and ~\ref{main_ran_sto}, we first need to provide some important technical results.
The following simple result demonstrates a few identities related to $\underline x_i^t$ (cf. \eqref{def_underlx}) and $y^t$ (cf. \eqref{def_yt1} or \eqref{def_yt1_sto}).% and  that will be useful for the analysis of RGEM algorithm.

\vgap
\begin{lemma}\label{tech_exp}
	Let $x^t$ and $y^t$ be defined in \eqnok{def_xt1} and \eqnok{def_yt1} (or \eqref{def_yt1_sto}), respectively, and $\hat{\underline x}_i^t$ and $\hat y^t$ be defined as 
	\begin{align}
	\hat{\underline x}_i^t&=(1+\tau_t)^{-1}(x^t+\tau_t\underline x_i^{t-1}), \ i=1,\dots,m, \  t\ge 1,\label{def_hxt}\\
	\hat{y}_i^t&=
	\begin{cases}
	% \Pr_{\mathcal{Y}_i}(-x^t, y_i^{t-1},\tau_t)=
	\nabla f_i(\hat{\underline x}_i^t), \ &\text{if $y^t$ is defined in \eqref{def_yt1}},\\
	\tfrac{1}{B_t}\tsum_{j=1}^{B_t}G_i(\hat {\underline x}_i^t,\xi_{i,j}^t), \ &\text{if $y^t$ is defined in \eqref{def_yt1_sto}},  
	\end{cases} \ i=1,\dots,m, \  t\ge 1, \label{def_hyt}
	\end{align}
	respectively. Then we have, for any $i = 1, \ldots, m$ and $t = 1, \ldots, k$,
	\begin{align*}
	\bbe_t[y_i^t] &= \tfrac{1}{m}\hat y_i^t +(1-\tfrac{1}{m})y_i^{t-1},\\
	% \bbe_t[\|y^{t-1}_i-y_i^t\|^2] &= \tfrac{1}{m}\|y^{t-1}_i-\hat y_i^t\|^2, \\
	\bbe_t[\underline x_i^t] &= \tfrac{1}{m} \hat{\underline x}_i^{t} + (1-\tfrac{1}{m})\underline x_i^{t-1},\\
	\bbe_t[f_i(\underline x_i^t)] &= \tfrac{1}{m} f_i(\hat{\underline x}_i^{t}) + (1-\tfrac{1}{m}) f_i(\underline x_i^{t-1}),\\
	\bbe_t[\|\nabla f_i({\underline x}^{t}_i)-\nabla f_i(\underline x_i^{t-1})\|_*^2] &= \tfrac{1}{m}\|\nabla f_i(\hat{\underline x}^{t}_i)-\nabla f_i(\underline x_{i}^{t-1})\|_*^2, 
	\end{align*}
	where $\bbe_t$ denotes the conditional expectation w.r.t. $i_t$ given $i_1, \ldots, i_{t-1}$ when $y^t$ is defined in \eqref{def_yt1}, and w.r.t. $i_t$ given $i_1, \ldots, i_{t-1}, \xi_1^t,\ldots,\xi_m^t$ when $y^t$ is defined in \eqref{def_yt1_sto}, respectively.
\end{lemma} 

\begin{proof}
	This first equality follows immediately from
	the facts that $\Prob_t \{y_i^t = \hat y_i^t\} = \Prob_t\{ i_t = i\}
	= \tfrac{1}{m}$ and $\Prob_t \{y_i^t = y_i^{t-1}\} = 1 - \tfrac{1}{m}$. Here $\Prob_t$ denotes the
	conditional probability w.r.t. $i_t$ given $i_1, \ldots, i_{t-1}$ when $y^t$ is defined in \eqref{def_yt1} and w.r.t $i_t$ given $i_1, \ldots, i_{t-1}, \xi_1^t,\ldots,\xi_m^t$ when $y^t$ is defined in \eqref{def_yt1_sto}, respectively. 
	Similarly, we can prove the rest equalities.
\end{proof}

\vgap
% Given a pair of feasible solution $\bar z=(\bar x,\frac{1}{m}\tsum_{i=1}^m \bar y_i)$ and $z=(x,\frac{1}{m}\tsum_{i=1}^m y_i)$ of the saddle point problem defined in \eqref{spp}, we define the primal-dual gap function $Q(\bar z, z)$ by
% \beq\label{def_Q}
% Q(\bar z, z):= \langle \bar x, \tfrac{1}{m}\tsum_{i=1}^m y_i\rangle - \frac{1}{m}\tsum_{i=1}^m J_i(y_i) + \mu w(\bar x) - [\langle x,\frac{1}{m}\tsum_{i=1}^m \bar y_i\rangle -\frac{1}{m}\tsum_{i=1}^m J_i(\bar y_i) +\mu w(x)].
% \eeq
We define the following function $Q$ to help us analyze the convergence properties of RGEM. Let $\underline x, x\in X$ be two feasible solutions of \eqref{cp} (or \eqref{sp}), we define the corresponding $Q(\underline x, x)$ by
\beq\label{def_Q}
Q(\underline x, x):= \langle \nabla f(x), \underline x -x\rangle + \mu w(\underline x) -\mu w(x).
\eeq
It is obvious that if we fix $x= x^*$, an optimal solution of \eqref{cp} (or \eqref{sp}), by the convexity of $w$ and the optimality condition of $x^*$, for any feasible solution $\underline x$, we can conclude that 
\begin{align*}
Q(\underline{x}, x^*)
&\ge \langle \nabla f(x^*) +\mu w'(x^*), \underline x -x^*\rangle  \ge 0.
\end{align*}
Moreover, observing that $f$ is smooth, we conclude that
% let 
% $\underline x^k$ be defined as in \eqref{def_ergodic_m}, and 
% $\underline y^k = (\tsum_{t=1}^k\theta_t)^{-1}\tsum_{t=1}^k\theta_t (\tfrac{1}{m}\tsum_{i=1}^m y^t_i)$,
% $x^*$ be an optimal solution of \eqref{cp}, we obtain
\begin{align}\label{psi_Q}
Q(\underline{x},x^*)
= f(x^*) + \langle \nabla f(x^*), \underline{x} - x^*\rangle + \mu w(\underline{x}) -\psi(x^*)
\ge - \tfrac{L_f}{2}\|\underline{x} -x^*\|^2 + \psi(\underline{x})- \psi(x^*).
\end{align}
The following lemma establishes an important relationship regarding $Q$. 

\vgap
\begin{lemma}\label{tech1_Q}
Let $x^t$ be defined in \eqref{def_xt1}, and $x \in X$ be any feasible solution of \eqref{cp} or \eqref{sp}. Suppose that $\tau_t$ in RGEM satisfy
\beq\label{tau_theta}
\theta_t(m(1+\tau_t)-1)=\theta_{t-1}m(1+\tau_{t-1}), \ \ t=2,\dots,k,
\eeq
for some $\theta_t\ge 0, \ t = 1,\ldots, k$. Then, we have
\begin{align}\label{rel_Q}
\tsum_{t=1}^k \theta_t\bbe[ Q(x^t,x)] 
&\le \theta_k(1+\tau_k)\tsum_{i=1}^m \bbe[f_i(\underline x_i^k)] + \tsum_{t=1}^k\theta_t\bbe[\mu w(x^t)-\psi(x)]\nn\\
&\quad-\theta_1(m(1+\tau_1)-1)[\langle x^0- x, \nabla f(x)\rangle +f(x)]. 
\end{align}
\end{lemma}
\begin{proof}
In view of the definition of $Q$ in \eqref{def_Q}, we have
\begin{align}
Q(x^t, x)
& = \tfrac{1}{m}\tsum_{i=1}^m\langle \nabla f_i(x), x^t -x\rangle + \mu w(x^t)- \mu w(x)\nn\\
&\varstackrel{\eqref{def_hxt}}{=} \tfrac{1}{m}\tsum_{i=1}^m [(1+\tau_t)\langle \underline{\hat x}_i^t -x,\nabla f_i(x)\rangle -\tau_t \langle \underline{x}_{i}^{t-1} - x,\nabla f_i(x)\rangle] + \mu w(x^t) -\mu w(x). \nn
\end{align}
% \begin{align*}
% Q(z^t, z)&\le  \langle x^t, \tfrac{1}{m}\tsum_{i=1}^m y_i\rangle - \frac{1}{m}\tsum_{i=1}^m J_i(y_i) + \mu w(x^t) - [\max_{y_i\in \mathcal{Y}_i}\{\langle x,\frac{1}{m}\tsum_{i=1}^m y_i\rangle -\frac{1}{m}\tsum_{i=1}^m J_i(y_i)\} +\mu w(x)]\\
% &\le \langle x^t -x, \tfrac{1}{m}\tsum_{i=1}^m y_i\rangle + f(x) + \mu w(x^t) -\psi(x)\\
% &\stackrel{\eqref{def_hxt}}{=} \tfrac{1}{m}\tsum_{i=1}^m [(1+\tau_1)\langle \underline{\hat x}_i^t -x,y_i\rangle -\tau_t \langle \underline{x}_{i}^{t-1} - x,y_i\rangle] + \mu w(x^t) -\mu w(x). 
% \end{align*}
Taking expectation on both sides of the above relation over $\{i_1, \ldots, i_k\}$, and using Lemma~\ref{tech_exp}, we obtain
\begin{align*}
\bbe[Q(x^t,x)]
% \footnote{Since star}
= \tsum_{i=1}^m \bbe[(1+\tau_t)\langle \underline{x}_i^t -x,\nabla f_i(x)\rangle - ((1+\tau_t)-\tfrac{1}{m}) \langle \underline{x}_{i}^{t-1} -x,\nabla f_i(x)\rangle] + \bbe[\mu w(x^t) -\mu w(x)].
\end{align*}
Multiplying both sides of the above inequality by $\theta_t$, and summing up the resulting inequalities from $t =1$ to $k$, we conclude that
\begin{align*}
\tsum_{t=1}^k \theta_t \bbe[Q(x^t,x)] 
&= \tsum_{i=1}^m \tsum_{t=1}^k \bbe[\theta_t(1+\tau_t)\langle \underline{x}_i^t -x,\nabla f_i(x)\rangle - \theta_t((1+\tau_t)-\tfrac{1}{m}) \langle \underline{x}_{i}^{t-1} -x,\nabla f_i(x)\rangle]\\
&\quad + \tsum_{t=1}^k\theta_t\bbe[\mu w(x^t) -\mu w(x)].
% &\stackrel{\eqref{tau_theta}}{\le} \tsum_{i=1}^m \bbe[] 
\end{align*}
Note that by \eqref{tau_theta} and the fact that $\underline{x}_i^0 = x^0, \ i = 1, \ldots, m$, we have 
\begin{align}
\tsum_{t=1}^k\theta_t &= \tsum_{t=2}^k [\theta_{t}m(1+\tau_{t})-\theta_{t-1}m(1+\tau_{t-1})]+\theta_1 = \theta_km(1+\tau_k)-\theta_1(m(1+\tau_1)-1), \label{sum_theta}\\
% \eeq
% \begin{align*}
% \sum_{t=1}^k [\theta_t(1+\tau_1)&\langle \underline{x}_i^t,\nabla f(x)\rangle - \theta_t((1+\tau_t)-\tfrac{1}{m}) \langle \underline{x}_{i}^{t-1},\nabla f(x)\rangle]
% = \theta_k(1+\tau_k)\langle \underline{x}_i^k,\nabla f(x)\rangle - \theta_1((1+\tau_1)-\tfrac{1}{m}) \langle \underline{x}_{i}^{0},\nabla f(x)\rangle\nn\\
\tsum_{t=1}^k [\theta_t(1+\tau_1)&\langle \underline{x}_i^t -x,\nabla f_i(x)\rangle - \theta_t((1+\tau_t)-\tfrac{1}{m}) \langle \underline{x}_{i}^{t-1} -x,\nabla f_i(x)\rangle]\nn\\
&= \theta_k(1+\tau_k)\langle \underline{x}_i^k -x,\nabla f_i(x)\rangle - \theta_1((1+\tau_1)-\tfrac{1}{m}) \langle x^{0} -x,\nabla f_i(x)\rangle, \ i = 1,\ldots,m.\nn
\end{align}
Combining the above three relations and using the convexity of $f_i$, we obtain
\begin{align*}
\tsum_{t=1}^k\theta_t\bbe[Q(x^t,x)]
% &\le \tsum_{i=1}^m\bbe[\theta_k(1+\tau_k)\langle \underline{x}_i^k -x,\nabla f_i(x)\rangle -\theta_1((1+\tau_1)-\tfrac{1}{m}) \langle \underline{x}_{i}^{0} -x,\nabla f_i(x)\rangle]\\
% &\quad + \tsum_{t=1}^k\theta_t\bbe[\mu w(x^t) -\mu w(x)]\\
&\le \theta_k(1+\tau_k)\tsum_{i=1}^m \bbe[f_i(\underline x_i^k) - f_i(x)] - \theta_1(m(1+\tau_1)-1) \langle x^{0} -x,\nabla f(x)\rangle\\
&\quad  + \tsum_{t=1}^k\theta_t\bbe[\mu w(x^t) -\mu w(x)],
\end{align*}
which in view of \eqref{sum_theta} implies \eqref{rel_Q}. 
\end{proof}

% The following lemma presents a similar result to Lemma 4 in \cite{lan2015optimal}.   
% \begin{lemma}\label{tech2_Q}
% Let $\underline x^k$ be defined as in \eqref{def_ergodic_m}, and 
% % $\underline y^k = (\tsum_{t=1}^k\theta_t)^{-1}\tsum_{t=1}^k\theta_t (\tfrac{1}{m}\tsum_{i=1}^m y^t_i)$,
% $x^*$ be an optimal solution of \eqref{cp}. Then, we have
% \beq\label{psi_Q}
% \psi(\underline x^k) -\psi(x^*) \le Q(\underline{x}^k,x^*) + \tfrac{L_f}{2}\|\underline{x}^k- x^*\|^2.
% \eeq 
% \end{lemma}
% \begin{proof}
% In view of the definitions of $Q$ in \eqref{def_Q}, and the convexity and smoothness of $f$, we have
% \begin{align*}
% Q(\underline{x}^k,\underline{x},x^*)
% \ge f(x^*) + \langle \nabla f(x^*), \underline{x}^k - x^*\rangle + \mu w(\underline{x}^k) -\psi(x^*)
% \ge f(\underline{x}^k) - \tfrac{L_f}{2}\|\underline{x}^k -x^*\|^2 + \mu w(\underline{x}^k) - \psi(x^*),
% \end{align*}
% while immediate implies \eqref{psi_Q}.
% \end{proof}

\vgap
\subsection{Convergence analysis of RGEM for deterministic finite-sum optimization}\label{sec_rws}
We now prove the main convergence properties for RGEM to solve \eqref{cp}. 
% Since it starts with $y_i^0=\nabla f_i(\underline x_i^0), \ i=1,\ldots,m$, and only updates the corresponding $i_t$-block of $(\underline{x}_i^t, y_i^t)$, $i=1,\ldots,m$, at the $t$-th iteration, we always have 
% \beq\label{yt_def}
% y_i^t=\nabla f_i(\underline x_i^t), \ i=1,\ldots,m.
% \eeq 
Observe that RGEM starts with $y^0= \0b$ and only updates the corresponding $i_t$-block of $(\underline{x}_i^t, y_i^t)$, $i=1,\ldots,m$, according to \eqref{def_underlx} and \eqref{def_yt1}, respectively. Therefore, for $y^t$ generated by RGEM, we have 
\beq\label{yt_relation}
y_i^t =
\begin{cases}
	\0b, & \text{if the $i$-th block has never been updated for the first $t$ iterations,}\\
	\nabla f_i(\underline x_i^t), &\text{o.w.}
\end{cases}
\eeq 
% Observe that this method reduces to GEM whenever $m=1$. 
% When no randomness presents, i.e., $m=1$ and $L_i=L_f$, Theorem~\ref{main_proper} reduces to Theorem~\ref{main_deter_pro}.
Throughout this subsection, we assume that there exists $\sigma_0\ge 0$ which is the upper bound of the initial gradients, i.e., \eqref{def_sigma} holds.
Proposition~\ref{main_proper} below establishes some general convergence properties of RGEM for solving strongly convex problems. 
% \todo{add something, change the following theorem into proposition}

\vgap
\begin{proposition}\label{main_proper}
	Let $x^t$ and $\underline x^k$ be defined as in \eqref{def_xt1} and \eqref{def_ergodic_m}, respectively, and $x^*$ be an optimal solution of \eqref{cp}. Under the assumption that there exists $\sigma_0$ satisfying \eqref{def_sigma}, and suppose 
	% that RGEM uses Type I initialization and 
	that $\{\eta_t\}$, $\{\tau_t\}$, and $\{\alpha_t\}$ in RGEM satisfy \eqref{tau_theta} and 
	\begin{align}
	m\theta_{t-1}&=\alpha_t\theta_t, \ \ t\ge2, \label{alpha_theta}\\
	\theta_t\eta_t&\le \theta_{t-1}(\mu+\eta_{t-1}), \ \ t \ge2, \label{eta_theta}\\
	% \theta_t(m(1+\tau_t)-1)&=\theta_{t-1}m(1+\tau_{t-1}), \ \ t=2,\dots,k, \label{tau_theta}\\
	% \alpha_tL_i&\le m\tau_{t-1}\eta_t, \ \ i=1,\dots,m; t\ge2, \label{alpha_tau_eta}\\
	% 2L_{i}&\le \tau_k(\mu+\eta_k),\ \ i=1,\dots,m, \label{L_tau_eta}
	% \end{align}
	% \begin{align}
		2\alpha_tL_i&\le m\tau_{t-1}\eta_t, \ \ i=1,\dots,m; \ t\ge2, \label{alpha_tau_eta_rws}\\
		4L_{i}&\le \tau_k(\mu+\eta_k),\ \ i=1,\dots,m, \label{L_tau_eta_rws}
		\end{align}
		for some $\theta_t\ge 0$, $t=1,\dots,k$. Then, for any $k\ge 1$, we have
		% then for any $k\ge 1$, we have
		\begin{align}\label{main_bnd_rws}
		\bbe[Q(\underline{x}^k,x^*)] 
	    &~\le (\tsum_{t=1}^k\theta_t)^{-1}\tilde \Delta_{0,\sigma_0},\nn\\
	    \bbe[P(x^k,x^*)]&~\le \tfrac{2\tilde \Delta_{0,\sigma_0}}{\theta_k(\mu+\eta_k)},
		% \theta_km(1+\tau_k)&\bbe[f(\underline x^k)-\psi(x)]+\tfrac{\theta_k(\mu+\eta_k)}{2}\bbe[P(x^k,x)]+\tsum_{t=1}^k\theta_t\bbe[\mu w(x^t)]\nn\\
		% &\le \theta_1\eta_1P(x^0,x)+\theta_1(m(1+\tau_1)-1)[f(x^0)-\psi(x)]+\tsum_{t=1}^k(\tfrac{m-1}{m})^{t-1}\tfrac{2\theta_t\alpha_{t+1}}{m\eta_{t+1}}\sigma^2.
		\end{align} 
		where
		\beq\label{def_tDS}
		\tilde \Delta_{0,\sigma_0}:= \theta_1(m(1+\tau_1)-1)(\psi(x^0)-\psi^*)+\theta_1\eta_1P(x^0,x^*) +\tsum_{t=1}^k(\tfrac{m-1}{m})^{t-1}\tfrac{2\theta_t\alpha_{t+1}}{m\eta_{t+1}}\sigma_0^2.
		\eeq
	% 	with $\tilde \Delta_0$ defined in \eqref{def_tDelta}.
	% for some $\theta_t\ge 0$, $t=1,\dots,k$. Then, for any $k\ge 1$, we have
	% \begin{align}\label{main_bnd}
	% \theta_km(1+\tau_k)\bbe[f(\underline x^k)-\psi(x)] +&\tfrac{\theta_k(\mu+\eta_k)}{2}\bbe[P(x^k,x)]+\tsum_{t=1}^k\theta_t\bbe[\mu w(x^t)]\nn\\
	% &~\le \theta_1(m(1+\tau_1)-1)[f(x^0)-\psi(x)]+\theta_1\eta_1P(x^0,x).
	% \bbe[Q(\underline{x}^k,x^*)] 
	% &~\le (\tsum_{t=1}^k\theta_t)^{-1}\tilde \Delta_0,\nn \\
	% \bbe[P(x^k,x^*)]&~\le \tfrac{2\tilde \Delta_0}{\theta_k(\mu+\eta_k)},
	% \end{align} 
	% where 
	% \beq\label{def_tDelta}
	% \tilde \Delta_{0}:=\theta_1(m(1+\tau_1)-1)(\psi(x^0)-\psi^*)+\theta_1\eta_1P(x^0,x^*).
	% \eeq
\end{proposition}
\begin{proof}
	In view of the definition of $x^t$ in \eqnok{def_xt1} and Lemma~\ref{opt_condi}, we have,
	\beq\label{primal_opt_s}
	\langle x^t-x, \tfrac{1}{m}\tsum_{i=1}^m\tilde y^t_i\rangle+\mu w(x^t)-\mu w(x) \le \eta_tP(x^{t-1},x)-(\mu+\eta_t)P(x^t,x)-\eta_tP(x^{t-1},x^t).
	\eeq
	Moreover, using the definition of $\psi$ in \eqref{cp}, the convexity of $f_i$, and the fact that $\hat y_i^t =\nabla f_i(\hat{\underline x}_i^t)$ (see \eqref{def_hyt} with $y^t$ defined in \eqref{def_yt1}), we obtain
	\begin{align}
	\tfrac{1+\tau_t}{m}\tsum_{i=1}^mf_i(\hat{\underline x}_i^t)+\mu w(x^t)-\psi(x)
	&\le \tfrac{1+\tau_t}{m}\tsum_{i=1}^mf_i(\hat{\underline x}_i^t)+\mu w(x^t)-\mu w(x)-\tfrac{1}{m}\tsum_{i=1}^m[f_i(\hat{\underline x}_i^t)+\la \hat y_i^t,x-\hat{\underline x}_i^t\ra]\nn\\
	&= \tfrac{\tau_t}{m}\tsum_{i=1}^m[f_i(\hat{\underline x}_i^t)+\la \hat y_i^t,\underline x_i^{t-1}-\hat{\underline x}_i^t\ra]+\mu w(x^t)-\mu w(x)-\tfrac{1}{m}\tsum_{i=1}^m\la \hat y_i^t,x-x^t\ra\nn\\
	&\le -\tfrac{\tau_t}{2m}\tsum_{i=1}^m\tfrac{1}{L_i}\|\nabla f_i(\hat{\underline x}_i^t)-\nabla f_i(\underline x_i^{t-1})\|_*^2
	+\tfrac{\tau_t}{m}\tsum_{i=1}^mf_i(\underline x_i^{t-1})\nn\\
	&\quad +\mu w(x^t)-\mu w(x)-\tfrac{1}{m}\tsum_{i=1}^m\la \hat y_i^t,x-x^t\ra\nn\\
	&\le -\tfrac{\tau_t}{2m}\tsum_{i=1}^m\tfrac{1}{L_i}\|\nabla f_i(\hat{\underline x}_i^t)-\nabla f_i(\underline x_i^{t-1})\|_*^2 +\tfrac{\tau_t}{m}\tsum_{i=1}^mf_i(\underline x_i^{t-1})\nn\\
	&\quad + \la x^t-x, \tfrac{1}{m} \tsum_{i=1}^m[\hat y_i^t-y_i^{t-1}-\alpha_t(y_i^{t-1}-y_i^{t-2})]\ra\nn\\ 
	&\quad+\eta_tP(x^{t-1},x)-(\mu+\eta_t)P(x^t,x)-\eta_tP(x^{t-1},x^t),\label{rec3}
	\end{align}
	where the first equality follows from the definition of $\hat {\underline x}_i^t$ in \eqref{def_hxt}, the second inequality follows from the smoothness of $f_i$ (see Theorem 2.1.5 in \cite{Nest04}) and \eqref{def_hyt}, and the last inequality follows from \eqref{primal_opt_s} and the definition of $\tilde y^t$ in \eqref{def_tildeyt}.
	Therefore, taking expectation on both sides of the above relation over $\{i_1,\dots,i_k\}$, and using Lemma~\ref{tech_exp}, we have
	\begin{align*}
	\bbe[(1+\tau_t)\tsum_{i=1}^mf_i(\underline x_i^t)+\mu w(x^t)-\psi(x)]
	& \le \bbe[-\tfrac{\tau_t}{2L_{i_t}}\|\nabla f_{i_t}(\underline x_{i_t}^{t})-\nabla f_{i_t}(\underline x_{i_t}^{t-1})\|_*^2 +\tfrac{1}{m}\tsum_{i=1}^m(m(1+\tau_t)-1)f_i(\underline{x}_i^{t-1})]\\
	&\quad +\bbe\{\langle x^t-x, \tfrac{1}{m}\tsum_{i=1}^m[m(y_i^t-y_i^{t-1})-\alpha_t(y_i^{t-1}-y_i^{t-2})]\rangle\}\\
	&\quad + \bbe[\eta_tP(x^{t-1},x)-(\mu+\eta_t)P(x^t,x)-\eta_tP(x^{t-1},x^t)].
	\end{align*}
	Multiplying both sides of the above inequality by $\theta_t$, and summing up the resulting inequalities from $t=1$ to $k$, we obtain 
	\begin{align}\label{rec1}
	\tsum_{t=1}^k\tsum_{i=1}^m\bbe[\theta_t(1+\tau_t)f_i(\underline x_i^t)]&+\tsum_{t=1}^k\theta_t\bbe[\mu w(x^t)-\psi(x)]\nn\\
	&\le \tsum_{t=1}^k\theta_t\bbe\left[-\tfrac{\tau_t}{2L_{i_t}}\|\nabla f_{i_t}(\underline x_{i_t}^{t})-\nabla f_{i_t}(\underline x_{i_t}^{t-1})\|_*^2+\tsum_{i=1}^m((1+\tau_t)-\tfrac{1}{m})f_i(\underline{x}_i^{t-1})\right]\nn\\
	&\quad + \tsum_{t=1}^k\tsum_{i=1}^m\theta_t\bbe[\langle x^t-x, y_i^t-y_i^{t-1}-\tfrac{\alpha_t}{m}(y_i^{t-1}-y_i^{t-2})\rangle]\nn\\
	&\quad + \tsum_{t=1}^k\theta_t\bbe[\eta_tP(x^{t-1},x)-(\mu+\eta_t)P(x^t,x)-\eta_tP(x^{t-1},x^t)].
	\end{align}
	Now by \eqref{alpha_theta}, and the facts that $y^{-1}=y^0$ and that we only update one block of $y^t$ (see \eqref{def_yt1}), we have
	\begin{align*}
	&\tsum_{t=1}^k\tsum_{i=1}^m\theta_t\bbe[\langle x^t-x, y_i^t-y_i^{t-1}-\tfrac{\alpha_t}{m}(y_i^{t-1}-y_i^{t-2})\rangle]\\
	&~ = \tsum_{t=1}^k\bbe[\theta_t\la x^t-x, y_{i_t}^t-y_{i_t}^{t-1}\ra-\tfrac{\theta_t\alpha_t}{m}\la x^{t-1}-x, y_{i_{t-1}}^{t-1}-y_{i_{t-1}}^{t-2}\ra]
	-\tsum_{t=2}^k\tfrac{\theta_t\alpha_t}{m}\bbe[\la x^t-x^{t-1},y_{i_{t-1}}^{t-1}-y_{i_{t-1}}^{t-2}\ra]\\
	&~ \varstackrel{\eqref{alpha_theta}}{=} \theta_k\bbe[\la x^k-x, y_{i_k}^k-y_{i_k}^{k-1}\ra  -\tsum_{t=2}^k\tfrac{\theta_t\alpha_t}{m}\bbe[\la x^t-x^{t-1},y_{i_{t-1}}^{t-1}-y_{i_{t-1}}^{t-2}\ra].
	\end{align*}
	Moreover, in view of \eqref{eta_theta}, \eqref{tau_theta}, and the fact that $\underline x_i^0=x^0,\ i=1,\ldots, m$, we obtain
	\begin{align*}
	\tsum_{t=1}^k \theta_t[\eta_tP(x^{t-1},x)-(\mu+\eta_t)P(x^t,x)]
	&\varstackrel{\eqref{eta_theta}}{\le} \theta_1\eta_1P(x^0,x)-\theta_k(\mu+\eta_k)P(x^k,x),\\
	\tsum_{t=1}^k\tsum_{i=1}^m \bbe[\theta_t(1+\tau_t)f_i(\underline x_i^t)-\theta_t((1+\tau_t)-\tfrac{1}{m})f_i(\underline{x}_i^{t-1})]
	&\varstackrel{\eqref{tau_theta}}{=}\tsum_{i=1}^m \bbe[\theta_k(1+\tau_k)f_i(\underline x_i^k)]- \theta_1(m(1+\tau_1)-1)f(x^0)
	% &\stackrel{\text{Lemma~\ref{tech_output_f}}}{\ge}\theta_km(1+\tau_k)\bbe[f(\underline x^k)]-\theta_1(m(1+\tau_1)-1)f(x^0),\\
	% \tsum_{t=1}^k\theta_t\stackrel{\eqref{tau_theta}}{=} \tsum_{t=2}^k [\theta_{t}m(1+\tau_{t})-\theta_{t-1}m(1+\tau_{t-1})]+\theta_1 
	% &= \theta_km(1+\tau_k)-\theta_1(m(1+\tau_1)-1),
	\end{align*}
	which together with \eqref{rec1} and \eqref{yt_relation}
	% and \eqref{yt_def} (due to Type I initialization) 
	imply that
	\begin{align}\label{rec2}
	\theta_k(1+\tau_k)&\tsum_{i=1}^m\bbe[f_i(\underline x_i^k)]+\tsum_{t=1}^k\theta_t\bbe[\mu w(x^t)-\psi(x)] +\tfrac{\theta_k(\mu+\eta_k)}{2}\bbe[P(x^k,x)]\nn\\
	% \theta_km(1+\tau_k)&\bbe[f(\underline x^k)-\psi(x)]+\tfrac{\theta_k(\mu+\eta_k)}{2}P(x^k,x)+\tsum_{t=1}^k\theta_t\bbe[\mu w(x^t)]\nn\\
	&\le \theta_1(m(1+\tau_1)-1)f(x^0)+\theta_1\eta_1P(x^0,x)\nn\\
	&\quad +\tsum_{t=2}^k\bbe\left[-\tfrac{\theta_t\alpha_t}{m}\langle x^t-x^{t-1}, y_{i_{t-1}}^{t-1}-y_{i_{t-1}}^{t-2}\rangle-\theta_t\eta_tP(x^{t-1},x^t)-\tfrac{\theta_{t-1}\tau_{t-1}}{2L_{i_{t-1}}}\|y_{i_{t-1}}^{t-1}-\nabla f_{i_{t-1}}(\underline x_{i_{t-1}}^{t-2})\|_*^2\right]\nn\\
	&\quad + \theta_k\bbe\left[\langle x^k-x,y_{i_k}^k-y_{i_k}^{k-1}\rangle-\tfrac{(\mu+\eta_k)}{2}P(x^k,x)-\tfrac{\tau_{k}}{2L_{i_{k}}}\|y_{i_{k}}^{k}-\nabla f_{i_{k}}(\underline x_{i_{k}}^{k-1})\|_*^2\right].
	% \label{rec3}
	\end{align}
% 	\begin{align}\label{rec2}
% 	\theta_k(1+\tau_k)&\tsum_{i=1}^m\bbe[f_i(\underline x_i^k)]+\tsum_{t=1}^k\theta_t\bbe[\mu w(x^t)-\psi(x)]\nn\\
% 	&\le \theta_1(m(1+\tau_1)-1)f(x^0)+\theta_1\eta_1P(x^0,x)\nn\\
% 	&\quad + \bbe\left[\theta_k\langle x^k-x,y_{i_k}^k-y_{i_k}^{k-1}\rangle-\theta_k(\mu+\eta_k)P(x^k,x)-\tfrac{\theta_k\tau_k}{2L_{i_k}}\|y_{i_k}^k-y_{i_k}^{k-1}\|_*^2\right]\nn\\
% 	&\quad +\tsum_{t=2}^k\bbe[-\tfrac{\theta_t\alpha_t}{m}\langle x^t-x^{t-1}, y_{i_{t-1}}^{t-1}-y_{i_{t-1}}^{t-2}\rangle-\theta_t\eta_tP(x^{t-1},x^t)-\tfrac{\theta_{t-1}\tau_{t-1}}{2L_{i_{t-1}}}\|y_{i_{t-1}}^{t-1}-y_{i_{t-1}}^{t-2}\|_*^2].
% %	&\le\bbe\left[\left(\tfrac{\theta_kL_{i_k}}{2\tau_k}-\tfrac{\theta_k(\mu+\eta_k)}{4}\right)\|x^k-x\|^2\right]+\tsum_{t=2}^k\bbe\left[\left(\tfrac{\theta_t\alpha_tL_{i_{t-1}}}{2m\tau_{t-1}}-\tfrac{\theta_t\eta_t}{2}\right)\|x^{t-1}-x^{t}\|^2\right]\nn\\
% %	&\quad +\theta_1\eta_1P(x^0,x)-\tfrac{\theta_k(\mu+\eta_k)}{2}\bbe[P(x^k,x)]\nn,
% 	\end{align}
	By the strong convexity of $P(\cdot,\cdot)$ in \eqref{P_strong}, the simple relations that $b\langle u,v\rangle-a\|v\|^2/2\le b^2\|u\|^2/(2a), \forall a>0$ and $\|a+b\|^2\le 2\|a\|^2 + 2\|b\|^2$, and the conditions in \eqnok{alpha_tau_eta_rws} and \eqnok{L_tau_eta_rws}, we have
	\begin{align*}
	\tsum_{t=2}^k&\left[-\tfrac{\theta_t\alpha_t}{m}\langle x^t-x^{t-1}, y_{i_{t-1}}^{t-1}-y_{i_{t-1}}^{t-2}\rangle-\theta_t\eta_tP(x^{t-1},x^t)-\tfrac{\theta_{t-1}\tau_{t-1}}{2L_{i_{t-1}}}\|y_{i_{t-1}}^{t-1}-\nabla f_{i_{t-1}}(\underline x_{i_{t-1}}^{t-2})\|_*^2\right]\nn\\
	&\varstackrel{\eqref{P_strong}}{\le}\tsum_{t=2}^k\left[-\tfrac{\theta_t\alpha_t}{m}\langle x^t-x^{t-1}, y_{i_{t-1}}^{t-1}-y_{i_{t-1}}^{t-2}\rangle-\tfrac{\theta_t\eta_t}{2}\|x^{t-1}-x^t\|^2-\tfrac{\theta_{t-1}\tau_{t-1}}{2L_{i_{t-1}}}\|y_{i_{t-1}}^{t-1}-\nabla f_{i_{t-1}}(\underline x_{i_{t-1}}^{t-2})\|_*^2\right]\\
	&\le \tsum_{t=2}^k\left[\tfrac{\theta_{t-1}\alpha_t}{2m\eta_t}\|y_{i_{t-1}}^{t-1}-y_{i_{t-1}}^{t-2}\|_*^2-\tfrac{\theta_{t-1}\tau_{t-1}}{2L_{i_{t-1}}}\|y_{i_{t-1}}^{t-1}-\nabla f_{i_{t-1}}(\underline x_{i_{t-1}}^{t-2})\|_*^2\right]\nn\\
	&\le
	\tsum_{t=2}^k\left[\left(\tfrac{\theta_{t-1}\alpha_t}{m\eta_t}-\tfrac{\theta_{t-1}\tau_{t-1}}{2L_{i_{t-1}}}\right)\|y_{i_{t-1}}^{t-1}-\nabla f_{i_{t-1}}(\underline x_{i_{t-1}}^{t-2})\|_*^2 +\tfrac{\theta_{t-1}\alpha_t}{m\eta_t}\|\nabla f_{i_{t-1}}(\underline x_{i_{t-1}}^{t-2})-y_{i_{t-1}}^{t-2}\|_*^2\right]\nn\\
	&\varstackrel{\eqref{alpha_tau_eta_rws}}{\le }\tsum_{t=2}^k\tfrac{\theta_{t-1}\alpha_t}{m\eta_t}\left[\|\nabla f_{i_{t-1}}(\underline x_{i_{t-1}}^{t-2})-y_{i_{t-1}}^{t-2}\|_*^2\right],
	\end{align*}
and similarly,
	\begin{align*}
	\theta_k&\left[\langle x^k-x,y_{i_k}^k-y_{i_k}^{k-1}\rangle-\tfrac{(\mu+\eta_k)}{2}P(x^k,x)-\tfrac{\tau_{k}}{2L_{i_{k}}}\|y_{i_{k}}^{k}-\nabla f_{i_{k}}(\underline x_{i_{k}}^{k-1})\|_*^2\right]\\
%	&\le\bbe\left[\tfrac{\theta_k}{\mu+\eta_k}\|y_{i_k}^k-y_{i_k}^{k-1}\|_*^2-\tfrac{\theta_k\tau_k}{2L_{i_k}}\|y_{i_{k}}^{k}-\nabla f_{i_{k}}(\underline x_{i_{k}}^{k-1})\|_*^2\right]
	&\le \tfrac{2\theta_k}{\mu+\eta_k}\left[\|\nabla f_{i_{k}}(\underline x_{i_{k}}^{k-1})-y_{i_{k}}^{k-1}\|_*^2\right]
	\le \tfrac{2\theta_k\alpha_{k+1}}{m\eta_{k+1}}\left[\|\nabla f_{i_{k}}(\underline x_{i_{k}}^{k-1})-y_{i_{k}}^{k-1}\|_*^2\right],
	\end{align*}
	where the last inequality follows from the fact that $m\eta_{k+1}\le \alpha_{k+1}(\mu+\eta_k)$ (induced from \eqref{alpha_theta} and \eqref{eta_theta}).
	 % we also have
	% \begin{align*}
	% &\theta_k\langle x^k-x,y_{i_k}^k-y_{i_k}^{k-1}\rangle-\tfrac{\theta_k(\mu+\eta_k)}{2}P(x^k,x)-\tfrac{\theta_k\tau_k}{2L_{i_k}}\|y_{i_k}^k-y_{i_k}^{k-1}\|_*^2\\
	% &~\le  \left(\tfrac{\theta_kL_{i_k}}{2\tau_k}-\tfrac{\theta_k(\mu+\eta_k)}{4}\right)\|x^k-x\|^2 \le 0,\\
	% % &\tsum_{t=2}^k\left[\tfrac{\theta_t\alpha_t}{m}\langle x^t-x^{t-1}, y_{i_{t-1}}^{t-1}-y_{i_{t-1}}^{t-2}\rangle-\theta_t\eta_tP(x^{t-1},x^t)-\tfrac{\theta_{t-1}\tau_{t-1}}{2L_{i_{t-1}}}\|y_{i_{t-1}}^{t-1}-y_{i_{t-1}}^{t-2}\|_*^2\right]\\
	% % &~\le \tsum_{t=2}^k\left(\tfrac{\theta_t\alpha_tL_{i_{t-1}}}{2m\tau_{t-1}}-\tfrac{\theta_t\eta_t}{2}\right)\|x^{t-1}-x^{t}\|^2 \le 0
	% \end{align*} 
	% Using the above relations in \eqref{rec2}, we obtain
	Therefore, combing the above three relations, 
%	and the fact that $m\eta_{k+1}\le \alpha_{k+1}(\mu+\eta_k)$ (induced from \eqref{alpha_theta} and \eqref{eta_theta}),
 we conclude that
	\begin{align}
	\theta_k(1+\tau_k)&\tsum_{i=1}^m\bbe[f_i(\underline x_i^k)]+\tsum_{t=1}^k\theta_t\bbe[\mu w(x^t)-\psi(x)] +\tfrac{\theta_k(\mu+\eta_k)}{2}\bbe[P(x^k,x)]\nn\\
	&\le \theta_1(m(1+\tau_1)-1)f(x^0)+\theta_1\eta_1P(x^0,x) +\tsum_{t=1}^{k}\tfrac{2\theta_t\alpha_{t+1}}{m\eta_{t+1}}\bbe[\|\nabla f_{i_{t}}(\underline x_{i_{t}}^{t-1})-y_{i_{t}}^{t-1}\|_*^2].\label{rws1}
	\end{align}
	We now provide a bound on $\bbe[\|\nabla f_{i_{t}}(\underline x_{i_{t}}^{t-1})-y_{i_{t}}^{t-1}\|_*^2]$. In view of \eqref{yt_relation}, we have
	\[
	\|\nabla f_{i_{t}}(\underline x_{i_{t}}^{t-1})-y_{i_{t}}^{t-1}\|_*^2=
	\begin{cases}
	\|\nabla f_{i_{t}}(\underline x_{i_{t}}^{t-1})\|_*^2, &\text{if the $i_t$-th block has never been updated until iteration $t$;}\\
	0, &\text{o.w.}
	\end{cases}
	\]
	Let us denote event \({\cal B}_{i_t}:=\{\text{the $i_t$-th block has never been updated until iteration $t$}\}\), for all $t=1,\ldots,k$, we have 
	\[
	\bbe[\|\nabla f_{i_{t}}(\underline x_{i_{t}}^{t-1})-y_{i_{t}}^{t-1}\|_*^2]=
	\bbe[\|\nabla f_{i_{t}}(\underline x_{i_{t}}^{t-1})\|_*^2|{\cal B}_{i_t}]\Prob\{{\cal B}_{i_t}\}\le \left(\tfrac{m-1}{m}\right)^{t-1}\sigma_0^2,
	\] 
	where the last inequality follows from the definitions of ${\cal B}_{i_t}$, $\underline x_i^t$ in \eqref{def_underlx} and $\sigma_0^2$ in \eqref{def_sigma}. 
	% \[
	% \theta_k(1+\tau_k)\tsum_{i=1}^m\bbe[f_i(\underline x_i^k)]+\tsum_{t=1}^k\theta_t\bbe[\mu w(x^t)-\psi(x)]\nn\\
	% \le \theta_1(m(1+\tau_1)-1)f(x^0)+\theta_1\eta_1P(x^0,x)- \tfrac{\theta_k(\mu+\eta_k)}{2}\bbe[P(x^k,x^*)].
	% \]
	Fixing $x = x^*$, and using the above result in \eqref{rws1}, we then conclude from \eqref{rws1} and Lemma~\ref{tech1_Q} that 
	\begin{align*}
	0\le \tsum_{t=1}^k\theta_t\bbe[Q(x^t,x^*)] 
	&~\le \theta_1(m(1+\tau_1)-1)[f(x^0) - \langle x^0-x^*, \nabla f(x^*)\rangle-f(x^*)]\\
	&\quad + \theta_1\eta_1P(x^0,x^*) +\tsum_{t=1}^k(\tfrac{m-1}{m})^{t-1}\tfrac{2\theta_t\alpha_{t+1}}{m\eta_{t+1}}\sigma_0^2 - \tfrac{\theta_k(\mu+\eta_k)}{2}\bbe[P(x^k,x^*)],
	% &~\le \theta_1(m(1+\tau_1)-1)[f(x^0) + \langle x^0-x^*, \mu \partial w(x^*)\rangle-f(x^*)]\\
	% &\quad + \theta_1\eta_1P(x^0,x^*) - \tfrac{\theta_k(\mu+\eta_k)}{2}\bbe[P(x^k,x^*)]\\
 %    &~\le \theta_1(m(1+\tau_1)-1)[\psi(x^0) -\psi(x^*)]\\&\quad + \theta_1\eta_1P(x^0,x^*) - \tfrac{\theta_k(\mu+\eta_k)}{2}\bbe[P(x^k,x^*)]\\
	\end{align*}
	which, in view of the relation $-\langle x^0-x^*,\nabla f(x^*)\rangle \le \langle x^0- x^*, \mu w'(x^*)\rangle \le \mu w(x^0)- \mu w(x^*)$ and the convexity of $Q(\cdot, x^*)$, implies the first result in \eqref{main_bnd_rws}.
	Moreover, we can also conclude from the above inequality that
	\begin{align*}
	\tfrac{\theta_k(\mu+\eta_k)}{2}\bbe[P(x^k,x^*)] 
	&~\le \theta_1(m(1+\tau_1)-1)[\psi(x^0)-\psi(x^*)] + \theta_1\eta_1P(x^0,x^*)+\tsum_{t=1}^k(\tfrac{m-1}{m})^{t-1}\tfrac{2\theta_t\alpha_{t+1}}{m\eta_{t+1}}\sigma_0^2,
	\end{align*}
	from which the second result in \eqref{main_bnd_rws} follows.
\end{proof}
 
\vgap

With the help of Proposition~\ref{main_proper}, we are now ready to prove Theorem~\ref{main_ran_sc}, which establishes the convergence properties of RGEM.
 % with Type II initialization. 
% Algorithm~\ref{alg_rpaged_rws}. 
In particular, Theorem~\ref{main_ran_sc} shows that 
% under the same parameter setting of the original RGEM and the assumption that the feasible set $X$ is bounded,
RGEM can achieve the optimal convergence rate as 
% \eqref{def_bnd_ran1} for smooth convex problems, as well as achieving 
${\cal O}\left\{\left(m+\sqrt{m\hat L/\mu}\right)\log 1/\epsilon\right\}$ for strongly convex problems. 
\vgap
\noindent{\bf Proof of Theorem~\ref{main_ran_sc}.}
%	On the other hand, when $\mu>0$, 
	Letting $\theta_t = \alpha^{-t},\ t=1,\dots,k$, we can easily check that parameter setting in \eqnok{constant_step_s_ran} with $\alpha$ defined in \eqref{def_nalpha} 
	satisfies conditions \eqref{tau_theta} and \eqnok{alpha_theta}-\eqnok{L_tau_eta_rws} stated in Proposition~\ref{main_proper}. 
	It then follows from \eqref{constant_step_s_ran} and \eqref{main_bnd_rws} that
	\begin{align*}
	\bbe[Q(\underline{x}^k,x^*)]
&\le \tfrac{\alpha^k}{1-\alpha^k}\left[\mu P(x^0,x^*) +\psi(x^0)-\psi^* + \tfrac{2m(1-\alpha)^2\sigma_0^2}{(m-1)\mu}\tsum_{t=1}^k\left(\tfrac{m-1}{m\alpha}\right)^{t}\right],\nn\\
\bbe[P(x^k,x^*)]
&\le2\alpha^k \left[P(x^0,x^*) +\tfrac{\psi(x^0)-\psi^*}{\mu}+\tfrac{2m(1-\alpha)^2\sigma_0^2}{(m-1)\mu^2}\tsum_{t=1}^k\left(\tfrac{m-1}{m\alpha}\right)^{t}\right],
	% \bbe[Q(\underline{x}^k,x^*)]
	% &\le \Delta_0 \alpha^k/(1-\alpha^{k}), \nn \\ 
	% \bbe[P(x^k,x^*)]
	% % &~\le \tfrac{2}{\theta_k(\mu +\eta_k)}[\theta_1(m(1+\tau_1)-1)[\psi(x^0)-\psi(x)]+\theta_1\eta_1P(x^0,x)]
	% &\le 2\Delta_0 \alpha^k/\mu, 
	\ \forall k\ge 1.
	\end{align*}
	Also observe that $\alpha\ge \tfrac{2m-1}{2m}$, we then have
	\[
	\tsum_{t=1}^k\left(\tfrac{m-1}{m\alpha}\right)^{t}\le \tsum_{t=1}^k\left(\tfrac{2(m-1)}{2m-1}\right)^{t}  \le 2(m-1).
	\]
Combining the above three relations and the fact that $m(1-\alpha)\le 1/2$, we have
\begin{align}\label{rel2}
\bbe[Q(\underline x^k,x^*)]&\le \tfrac{\alpha^k}{1-\alpha^k}\Delta_{0,\sigma_0},\nn\\
\bbe[P(x^k,x^*)]&\le 2\alpha^k\Delta_{0,\sigma_0}/\mu, \ \forall k\ge 1,
\end{align}
where $\Delta_{0,\sigma_0}$ is defined in \eqref{def_DeltaS}.
	% where $\Delta_0$ is defined in \eqref{def_Delta}.
	% In view of the fact that $m(1-\alpha)\le 1/2$, 
	The second relation immediately implies our bound in \eqref{ran_bnd_s1_rws}. Moreover, by the strong convexity of $P(\cdot,\cdot)$ in \eqref{P_strong} and \eqref{ran_bnd_s1_rws}, we have 
	\begin{align*}
	\tfrac{L_f}{2}\bbe[\|\underline x^k-x^*\|^2]
	&~\le \tfrac{L_f}{2}(\tsum_{t=1}^k\theta_t)^{-1}\tsum_{t=1}^k\theta_t \bbe[\|x^t-x^*\|^2]
	\stackrel{\eqref{P_strong}}{\le} L_f \tfrac{(1-\alpha)\alpha^k}{1-\alpha^{k}}\tsum_{t=1}^k \alpha^{-t} \bbe[P(x^t,x^*)]\\
	&~\varstackrel{\eqref{ran_bnd_s1_rws}}{\le} \tfrac{L_f(1-\alpha)\alpha^k}{1-\alpha^{k}}\tsum_{t=1}^k 
	% \tfrac{2\alpha^t}{\alpha^{3t/2}}
	\tfrac{2 \Delta_{0,\sigma_0}}{\mu}
	= \tfrac{2L_f(1-\alpha) \Delta_{0,\sigma_0} k\alpha^{k}}{\mu(1-\alpha^k)}.
	\end{align*}
	Combining the above relation with the first inequality in \eqref{rel2} and \eqref{psi_Q}, we obtain
	\begin{align*}
	\bbe[\psi(\underline{x}^k)-\psi(x^*)]\stackrel{\eqref{psi_Q}}{\le} \bbe[Q(\underline{x}^k,x^*)] + \tfrac{L_f}{2}\bbe[\|\underline{x}^k- x^*\|^2] \le \left(1+\tfrac{2L_f(1-\alpha)}{\mu}k\right)\tfrac{\Delta_{0,\sigma_0}\alpha^{k}}{1-\alpha^k}.  
	\end{align*}
	Observing that 
	% $1/(1-\alpha)\le 2\max\{m,\hat L/\mu\}$,
	\begin{align*}
	% \label{k_alpha}
	\tfrac{1}{1-\alpha}&\le 2\max\{m,\hat L/\mu\},\nn\\
(k+1)\tfrac{\alpha^{k}(1-\alpha)}{1-\alpha^k}
	&= \left(\tsum_{t=1}^{k}\tfrac{\alpha^{t}}{\alpha^t}+1\right)\tfrac{\alpha^{k}(1-\alpha)}{1-\alpha^k}	
	\le \left(\tsum_{t=1}^{k}\tfrac{\alpha^t}{\alpha^{3t/2}}+1\right)\tfrac{\alpha^{k}(1-\alpha)}{1-\alpha^k}\nn\\
	&\le \tfrac{1-\alpha^{k/2}}{\alpha^{k/2}(1-\alpha^{1/2})}\tfrac{\alpha^{k}(1-\alpha)}{1-\alpha^k} +\alpha^{k}
	\le 2\alpha^{k/2} + \alpha^k \le 3\alpha^{k/2},
\end{align*}
	we have 
	\begin{align}
	% \label{tighter_bnd_s2}
	\bbe[\psi(\underline{x}^k)-\psi(x^*)]
	% \stackrel{\eqref{psi_Q}}{\le} \bbe[Q(\underline{x}^k,x^*)] + \tfrac{L_f}{2}\|\underline{x}^k- x^*\|^2 
	&\le 2\max\left\{m,\tfrac{\hat L}{\mu}\right\}\tfrac{\Delta_{0,\sigma_0} (k+1)\alpha^{k}(1-\alpha)}{1-\alpha^k}
	\le 6\max\left\{m,\tfrac{\hat L}{\mu}\right\}\Delta_{0,\sigma_0}\alpha^{k/2}. \nn
	\end{align}

\vgap

\subsection{Convergence analysis of RGEM for stochastic finite-sum optimization}\label{sec_sto}
% \todo{edit!!}
% In this section, we consider the stochastic setting where only the noisy gradient information of each component function $f_i$, $i= 1,\ldots,m$, is available. 
Our goal in this section is to establish the convergence properties of RGEM for solving stochastic finite-sum optimization problems in \eqref{sp}.
For notation convenience, we use $\bbe_{[i_k]}$ for taking expectation over $\{i_1,\ldots,i_k\}$, $\bbe_{\xi}$ for expectations over $\{\xi^1,\ldots,\xi^k\}$, respectively, we use $\bbe$ to denote the expectations over all random variables. 

Note that the parameter $\{B_t\}$ in Algorithm~\ref{alg_rpaged_sto} denotes the batch size used to compute $y_{i_t}^t$ in \eqref{def_yt1_sto}. Since we now assume that $\|\cdot\|$ is associated with a certain inner product, it can be easily seen from \eqref{def_yt1_sto}, and the two assumptions we have for the stochastic gradients computed by $\SO$ oracle, i.e.,  \eqref{assmp_unbias} and \eqref{assump_var}, that
\beq\label{batch_g}
\bbe_{\xi}[y_{i_t}^t]= \nabla f_{i_t}(\underline{x}_{i_t}^t) \ \mbox{and} \ 
\bbe_{\xi}[\|y_{i_t}^t-\nabla f_{i_t}(\underline{x}_{i_t}^t)\|_*^2] \le \tfrac{\sigma^2}{B_t}, \ \forall i_t, t = 1,\ldots,k,
\eeq
and hence $y_{i_t}^t$ is an unbiased estimator of $\nabla f_{i_t}(\underline{x}_{i_t}^t)$. Moreover, for $y^t$ generated by Algorithm~\ref{alg_rpaged_sto}, we can see that
\beq\label{sto_yt}
y_i^t=
\begin{cases}
	\0b, & \text{if the $i$-th block has never been updated for the first $t$ iterations;}\\
	\tfrac{1}{B_l}\tsum_{j=1}^{B_l}G_i(\underline x_i^l,\xi_{i,j}^l), &\text{if the latest update happened at $l$-th iteration, for $1\le l\le t$.}
\end{cases}
\eeq
We first establish some general convergence properties for Algorithm~\ref{alg_rpaged_sto}.

\vgap
\begin{proposition}\label{main_sto_rgem}
Let $x^t$ and $\underline x^k$ be defined as in \eqref{def_xt1} and \eqref{def_ergodic_m}, respectively, and $x^*$ be an optimal solution of \eqref{sp}. Suppose that $\sigma_0$ and $\sigma$ are defined in \eqref{def_sigma} and \eqref{assump_var}, respectively, and $\{\eta_t\}$, $\{\tau_t\}$, and $\{\alpha_t\}$ in Algorithm~\ref{alg_rpaged_sto} satisfy \eqref{tau_theta}, \eqref{alpha_theta}， \eqref{eta_theta}, and \eqref{L_tau_eta_rws} for some $\theta_t\ge 0$, $t=1,\dots,k$. Moreover, if 
\beq\label{alpha_tau_eta_sto}
3\alpha_tL_i\le m\tau_{t-1}\eta_t, \ i=1,\ldots,m;t\ge2,
\eeq
then for any $k\ge 1$, we have
\begin{align}\label{main_bnd_sto}
\bbe[Q(\underline{x}^k,x^*)] 
	    &\le (\tsum_{t=1}^k\theta_t)^{-1}\tilde\Delta_{0,\sigma_0,\sigma},\nn\\
\bbe[P(x^k,x^*)]
        &\le \tfrac{2\tilde \Delta_{0,\sigma_0,\sigma}}{\theta_k(\mu+\eta_k)},
\end{align}
where 
\beq\label{def_tDSS}
\tilde \Delta_{0,\sigma_0,\sigma}:= \tilde\Delta_{0,\sigma_0} + \tsum_{t=2}^k \tfrac{3\theta_{t-1}\alpha_t\sigma^2}{2m\eta_tB_{t-1}} + \tsum_{t=1}^k\tfrac{2\theta_t\alpha_{t+1}}{m^2\eta_{t+1}}\tsum_{l=1}^{t-1}(\tfrac{m-1}{m})^{t-1-l}\tfrac{\sigma^2}{B_l},
\eeq
with $\tilde \Delta_{0,\sigma_0}$ defined in \eqref{def_tDS}.
\end{proposition}
\begin{proof}
Observe that in Algorithm~\ref{alg_rpaged_sto} $y^t$ is updated as in \eqref{def_yt1_sto}. Therefore, according to \eqref{def_hyt}, we have 
\[
	\hat y_i^t =\tfrac{1}{B_t}\tsum_{j=1}^{B_t}G_i(\hat{\underline x}_{i}^t,\xi_{i,j}^t), \ i =1,\ldots,m,\ t\ge 1,
\]
which together with the first relation in \eqref{batch_g} imply that $\bbe_{\xi}[\la \hat y_i^t, x-\hat{\underline x}_i^t\ra]=\bbe_{\xi}[\la \nabla f_i(\hat{\underline x}_i^t), x-\hat{\underline x}_i^t\ra]$.
% is different from RGEM (cf. Algorithm~\ref{alg_rpaged}) in how the iterates $y^t$ is updated. Since , 
Hence, we can rewrite \eqref{rec3} as
\begin{align*}
	\bbe_{\xi}[\tfrac{1+\tau_t}{m}\tsum_{i=1}^mf_i(\hat{\underline x}_i^t)+\mu w(x^t)-\psi(x)]
	&\le \bbe_{\xi}\left[\tfrac{1+\tau_t}{m}\tsum_{i=1}^mf_i(\hat{\underline x}_i^t)+\mu w(x^t)-\mu w(x)-\tfrac{1}{m}\tsum_{i=1}^m[f_i(\hat{\underline x}_i^t)+\la \nabla f_i(\hat{\underline x}_i^t),x-\hat{\underline x}_i^t\ra]\right]\\
	&= \bbe_{\xi}\left[\tfrac{1+\tau_t}{m}\tsum_{i=1}^mf_i(\hat{\underline x}_i^t)+\mu w(x^t)-\mu w(x)-\tfrac{1}{m}\tsum_{i=1}^m[f_i(\hat{\underline x}_i^t)+\la \hat y_i^t,x-\hat{\underline x}_i^t\ra]\right]\\
	% &= \tfrac{\tau_t}{m}\tsum_{i=1}^m[f_i(\hat{\underline x}_i^t)+\la \hat y_i^t,\underline x_i^{t-1}-\hat{\underline x}_i^t\ra]+\mu w(x^t)-\mu w(x)-\tfrac{1}{m}\tsum_{i=1}^m\la \hat y_i^t,x-x^t\ra\nn\\
	% &\le -\tfrac{\tau_t}{2m}\tsum_{i=1}^m\tfrac{1}{L_i}\|\nabla f_i(\hat{\underline x}_i^t)-\nabla f_i(\underline x_i^{t-1})\|_*^2
	% +\tfrac{\tau_t}{m}\tsum_{i=1}^mf_i(\underline x_i^{t-1})\nn\\
	% &\quad +\mu w(x^t)-\mu w(x)-\tfrac{1}{m}\tsum_{i=1}^m\la \hat y_i^t,x-x^t\ra\nn\\
	&\le \bbe_{\xi}\big[-\tfrac{\tau_t}{2m}\tsum_{i=1}^m\tfrac{1}{L_i}\|\nabla f_i(\hat{\underline x}_i^t)-\nabla f_i(\underline x_i^{t-1})\|_*^2 +\tfrac{\tau_t}{m}\tsum_{i=1}^mf_i(\underline x_i^{t-1})\nn\\
	&\quad\quad + \la x^t-x, \tfrac{1}{m} \tsum_{i=1}^m[\hat y_i^t-y_i^{t-1}-\alpha_t(y_i^{t-1}-y_i^{t-2})]\ra\nn\\ 
	&\quad\quad+\eta_tP(x^{t-1},x)-(\mu+\eta_t)P(x^t,x)-\eta_tP(x^{t-1},x^t)\big],
	\end{align*}
Following the same procedure as in the proof of Proposition~\ref{main_proper}, we obtain the following similar relation (cf. \eqref{rec2})
\begin{align*}
	\theta_k&(1+\tau_k)\tsum_{i=1}^m\bbe[f_i(\underline x_i^k)]+\tsum_{t=1}^k\theta_t\bbe[\mu w(x^t)-\psi(x)] +\tfrac{\theta_k(\mu+\eta_k)}{2}\bbe[P(x^k,x)]\nn\\
	% \theta_km(1+\tau_k)&\bbe[f(\underline x^k)-\psi(x)]+\tfrac{\theta_k(\mu+\eta_k)}{2}P(x^k,x)+\tsum_{t=1}^k\theta_t\bbe[\mu w(x^t)]\nn\\
	&\le \theta_1(m(1+\tau_1)-1)f(x^0)+\theta_1\eta_1P(x^0,x)\nn\\
	&~+\tsum_{t=2}^k\bbe\left[-\tfrac{\theta_t\alpha_t}{m}\langle x^t-x^{t-1}, y_{i_{t-1}}^{t-1}-y_{i_{t-1}}^{t-2}\rangle-\theta_t\eta_tP(x^{t-1},x^t)-\tfrac{\theta_{t-1}\tau_{t-1}}{2L_{i_{t-1}}}\|\nabla f_{i_{t-1}}(x_{i_{t-1}}^{t-1})-\nabla f_{i_{t-1}}(\underline x_{i_{t-1}}^{t-2})\|_*^2\right]\nn\\
	&~ + \theta_k\bbe\left[\langle x^k-x,y_{i_k}^k-y_{i_k}^{k-1}\rangle-\tfrac{(\mu+\eta_k)}{2}P(x^k,x)-\tfrac{\tau_{k}}{2L_{i_{k}}}\|\nabla f_{i_k}(x_{i_{k}}^{k})-\nabla f_{i_{k}}(\underline x_{i_{k}}^{k-1})\|_*^2\right].
	% \label{rec3}
	\end{align*}
	By the strong convexity of $P(\cdot,\cdot)$ in \eqref{P_strong}, the fact that $b\langle u,v\rangle-a\|v\|^2/2\le b^2\|u\|^2/(2a), \forall a>0$, and the Cauchy-Schwartz inequality, we have, for $t =2,\ldots,k$, 
	\begin{align*}
	&\bbe[-\tfrac{\theta_t\alpha_t}{m}\langle x^t-x^{t-1}, y_{i_{t-1}}^{t-1}-y_{i_{t-1}}^{t-2}\rangle-\theta_t\eta_tP(x^{t-1},x^t)-\tfrac{\theta_{t-1}\tau_{t-1}}{2L_{i_{t-1}}}\|\nabla f_{i_{t-1}}(x_{i_{t-1}}^{t-1})-\nabla f_{i_{t-1}}(\underline x_{i_{t-1}}^{t-2})\|_*^2]\\
	&\varstackrel{\eqref{P_strong}}{\le} \bbe[-\tfrac{\theta_t\alpha_t}{m}\langle x^t-x^{t-1}, y_{i_{t-1}}^{t-1}-\nabla f_{i_{t-1}}(\underline x_{i_{t-1}}^{t-1})+ \nabla f_{i_{t-1}}(\underline x_{i_{t-1}}^{t-1}) - \nabla f_{i_{t-1}}(\underline x_{i_{t-1}}^{t-2})+ \nabla f_{i_{t-1}}(\underline x_{i_{t-1}}^{t-2})-y_{i_{t-1}}^{t-2}\rangle]\\
	&\quad -\bbe\left[\tfrac{\theta_t\eta_t}{2}\|x^{t-1}-x^t\|^2 +\tfrac{\theta_{t-1}\tau_{t-1}}{2L_{i_{t-1}}}\|\nabla f_{i_{t-1}}(x_{i_{t-1}}^{t-1})-\nabla f_{i_{t-1}}(\underline x_{i_{t-1}}^{t-2})\|_*^2\right]\\
	&\le \bbe\left[\left(\tfrac{3\theta_{t-1}\alpha_t}{2m\eta_t}-\tfrac{\theta_{t-1}\tau_{t-1}}{2L_{i_{t-1}}}\right)\|\nabla f_{i_{t-1}}(x_{i_{t-1}}^{t-1})-\nabla f_{i_{t-1}}(\underline x_{i_{t-1}}^{t-2})\|_*^2\right]\\
	&\quad + \tfrac{3\theta_{t-1}\alpha_t}{2m\eta_t}\bbe\left[\|y_{i_{t-1}}^{t-1}-\nabla f_{i_{t-1}}(\underline x_{i_{t-1}}^{t-1})\|_*^2 + \|\nabla f_{i_{t-1}}(\underline x_{i_{t-1}}^{t-2})-y_{i_{t-1}}^{t-2}\|_*^2\right]\\
	&\varstackrel{\eqref{alpha_tau_eta_sto}}{\le}\tfrac{3\theta_{t-1}\alpha_t}{2m\eta_t}\bbe\left[\|y_{i_{t-1}}^{t-1}-\nabla f_{i_{t-1}}(\underline x_{i_{t-1}}^{t-1})\|_*^2 + \|\nabla f_{i_{t-1}}(\underline x_{i_{t-1}}^{t-2})-y_{i_{t-1}}^{t-2}\|_*^2\right].
	\end{align*}
	Similarly, we can also obtain
	\begin{align*}
	&\bbe\left[\langle x^k-x,y_{i_k}^k-y_{i_k}^{k-1}\rangle - \tfrac{(\mu+\eta_k)}{2}P(x^k,x)-\tfrac{\tau_{k}}{2L_{i_{k}}}\|f_{i_k}(x_{i_{k}}^{k})-\nabla f_{i_{k}}(\underline x_{i_{k}}^{k-1})\|_*^2\right]\\
	% &=\bbe_{\xi}\bbe_{[i_k]}\left[\langle x^k-x,y_{i_k}^k-\nabla f_{i_k}(\underline x_{i_k}^k) + \nabla f_{i_k}(\underline x_{i_k}^k) - \nabla f_{i_k}(\underline x_{i_k}^{k-1})+ \nabla f_{i_k}(\underline x_{i_k}^{k-1})-y_{i_k}^{k-1}\rangle\right]\\
	&\varstackrel{\eqref{batch_g},\eqref{P_strong}}{\le} \bbe\left[\langle x^k-x, \nabla f_{i_k}(\underline x_{i_k}^k) - \nabla f_{i_k}(\underline x_{i_k}^{k-1})+ \nabla f_{i_k}(\underline x_{i_k}^{k-1}) -y_{i_k}^{k-1}\rangle\right]\\
	&\quad - \bbe\left[\tfrac{(\mu+\eta_k)}{4}\|x^k-x\|^2+\tfrac{\tau_{k}}{2L_{i_{k}}}\|f_{i_k}(x_{i_{k}}^{k})-\nabla f_{i_{k}}(\underline x_{i_{k}}^{k-1})\|_*^2\right]\\
	&\le \bbe\left[\left(\tfrac{2}{\mu+\eta_k}-\tfrac{\tau_k}{2L_{i_{k}}}\right)\|\nabla f_{i_k}(x_{i_{k}}^{k})-\nabla f_{i_{k}}(\underline x_{i_{k}}^{k-1})\|_*^2 + \tfrac{2}{\mu+\eta_k}\|\nabla f_{i_k}(\underline x_{i_k}^{k-1}) -y_{i_k}^{k-1}\|_*^2\right] \\
	&\varstackrel{\eqref{L_tau_eta_rws}}{\le}\bbe\left[\tfrac{2}{\mu+\eta_k}\|\nabla f_{i_k}(\underline x_{i_k}^{k-1}) -y_{i_k}^{k-1}\|_*^2\right].
	\end{align*}
	% where the first inequality follows from \eqref{batch_g} and \eqref{P_strong}. 
	Combining the above three relations, and using the fact that $m\eta_{k+1}\le \alpha_{k+1}(\mu+\eta_k)$ (induced from \eqref{alpha_theta} and \eqref{eta_theta}), we have
	\begin{align*}
	\theta_k&(1+\tau_k)\tsum_{i=1}^m\bbe[f_i(\underline x_i^k)]+\tsum_{t=1}^k\theta_t\bbe[\mu w(x^t)-\psi(x)] +\tfrac{\theta_k(\mu+\eta_k)}{2}\bbe[P(x^k,x)]\nn\\
	&\le \theta_1(m(1+\tau_1)-1)f(x^0)+\theta_1\eta_1P(x^0,x)\\
	&\quad +\tsum_{t=2}^{k} \tfrac{3\theta_{t-1}\alpha_{t}}{2m\eta_{t}}\bbe[\|y_{i_{t-1}}^{t-1}-\nabla f_{i_{t-1}}(\underline{x}_{i_{t-1}}^{t-1})\|_*^2] + \tsum_{t=1}^k\tfrac{2\theta_t\alpha_{t+1}}{m\eta_{t+1}}\bbe[\|\nabla f_{i_{t}}(\underline x_{i_{t}}^{t-1})-y_{i_{t}}^{t-1}\|_*^2].
	\end{align*}
	Moreover, in view of the second relation in \eqref{batch_g}, we have
	\begin{align*}
	\bbe[\|y_{i_{t-1}}^{t-1}-\nabla f_{i_{t-1}}(\underline x_{i_{t-1}}^{t-1})\|_*^2] &\le \tfrac{\sigma^2}{B_{t-1}}, \ \forall t\ge 2.
	\end{align*}
	Let us denote ${\cal E}_{i_t,t}:=\max\{l:i_l = i_t, l<t\}$ with ${\cal E}_{i_t,t}=0$ denoting the event that the $i_t$-th block has never been updated until iteration $t$, 
	% \eqref{sto_yt} and \eqref{def_sigma}, we can conclude that
	% Note that for any 
	% % $\nabla f_{i_j}(\underline x_{i_j}^{j-1})$ and $\nabla f_{i_l}(\underline x_{i_l}^{l-1})$, their latest updates can't happen at the same iteration whenever 
	% $j\ne k$, ${\cal E}_{i_j,j}\ne {\cal E}_{i_k,k}$. In fact, without loss of generosity, if $j< k$ and ${\cal E}_{i_j,j} = {\cal E}_{i_k,k}=l< j$, we have $i_l=i_j=i_k$, which implies ${\cal E}_{i_k,k} \ge j$ yielding a contradiction. 
	we can also conclude that for any $t\ge 1$
	\begin{align*}
	\bbe[\|\nabla f_{i_{t}}(\underline x_{i_{t}}^{t-1})-y_{i_{t}}^{t-1}\|_*^2]
	 &= \tsum_{l=0}^{t-1}\bbe\left[\|\nabla f_{i_{l}}(\underline x_{i_{l}}^{l})-y_{i_{l}}^{l}\|_*^2|\{{\cal E}_{i_{t},t}=l\}\right] \prob\{{\cal E}_{i_{t},t}=l\}\\
	&\le (\tfrac{m-1}{m})^{t-1}\sigma_0^2 + \tsum_{l=1}^{t-1}\tfrac{1}{m}(\tfrac{m-1}{m})^{t-1-l}\tfrac{\sigma^2}{B_l},
	\end{align*}
	where the first term in the inequality corresponds to the case when the $i_t$-block has never been updated for the first $t-1$ iterations, and the second term represents that its latest update for the first $t-1$ iterations happened at the $l$-th iteration.
	Hence, using Lemma~\ref{tech1_Q} and following the same argument as in the proof of Proposition~\ref{main_proper}, we obtain our results in \eqref{main_bnd_sto}.
\end{proof}

\vgap
We are now ready to prove Theorem~\ref{main_ran_sto}, which establishes an optimal complexity bound (up to a logarithmic factor) on the number of calls to the $\SO$ oracle and a linear rate of convergence in terms of the communication complexity for solving problem \eqref{sp}.
% \begin{corollary}\label{main_ran_sto}
% 	Let $x^*$ be an optimal solution of \eqnok{cp}, $x^k$ and $\underline x^k$ be defined in \eqref{def_xt1} and \eqnok{def_ergodic_m}, respectively, and $\hat L=\max_{i=1,\dots,m}{L_i}$. Given the iteration limit $k$, let $\{\tau_t\}$, $\{\eta_t\}$ and $\{\alpha_t\}$ be set to \eqref{constant_step_s_ran} with $\alpha$ being set as \eqref{def_nalpha}, and we set 
% 	\beq\label{def_bt}
% 	B_t = \lceil k(1-\alpha)^2\alpha^{-t}\rceil, \ t =1,\ldots,k,
% 	\eeq
% 	then
% 	\begin{align}
% 		\bbe[P(x^k,x^*)]&\le \tfrac{2\alpha^k\tilde\Delta_{0,\sigma_0,\sigma}}{\mu} ,\label{ran_bnd_s1_sto}\\
% 		\bbe[\psi(\underline{x}^k)-\psi(x^*)]&\le 2\max\{m,\tfrac{\hat L}{\mu}\} \tilde \Delta_{0,\sigma_0,\sigma}(k+1)\alpha^{k},	\label{ran_bnd_s2_sto}
% 		\end{align}
% 		where 
% 		\beq\label{def_tDSS}
% 		\tilde \Delta_{0,\sigma_0,\sigma}:= (1-\alpha)\Delta_{0,\sigma_0,\sigma} = \mu P(x^0,x^*)+\psi(x^0)-\psi(x^*) + \tfrac{\sigma_0^2/m+5\sigma^2}{\mu},
% 		\eeq
% 		with $\Delta_{0,\sigma_0.\sigma}$ defined in \eqref{def_DSS}.
% \end{corollary}
% \begin{proof}

\vgap
\noindent{\bf Proof of Theorem~\ref{main_ran_sto}}
Let us set $\theta_t = \alpha^{-t}, \ t = 1,\ldots,k$. It is easy to check that the parameter setting in \eqref{constant_step_s_ran} with $\alpha$ defined in \eqref{def_nalpha} satisfies conditions \eqref{tau_theta}, \eqref{alpha_theta}, \eqref{eta_theta}, \eqref{L_tau_eta_rws}, and \eqref{alpha_tau_eta_sto} as required by Proposition~\ref{main_sto_rgem}.
By \eqref{constant_step_s_ran}, the definition of $B_t$ in \eqref{def_bt}, and the fact that $\alpha \ge \tfrac{2m-1}{2m} > (m-1)/m$, we have
\begin{align*}
\tsum_{t=2}^k \tfrac{3\theta_{t-1}\alpha_t\sigma^2}{2m\eta_tB_{t-1}}
&\le \tsum_{t=2}^k \tfrac{3\sigma^2}{2\mu(1-\alpha)k}
\le \tfrac{3\sigma^2}{2\mu(1-\alpha)},\\
\tsum_{t=1}^k\tfrac{2\theta_t\alpha_{t+1}}{m^2\eta_{t+1}}\tsum_{l=1}^{t-1}(\tfrac{m-1}{m})^{t-1-l}\tfrac{\sigma^2}{B_l}
&\le\tfrac{2\sigma^2}{\alpha\mu m(1-\alpha)k}\tsum_{t=1}^k(\tfrac{m-1}{m\alpha})^{t-1}\tsum_{l=1}^{t-1}(\tfrac{m\alpha}{m-1})^{l}\\
&\le \tfrac{2\sigma^2}{\mu (1-\alpha)m\alpha k}\tsum_{t=1}^k(\tfrac{m-1}{m\alpha})^{t-1}(\tfrac{m\alpha}{m-1})^{t-1}\tfrac{1}{1-(m-1)/(m\alpha)}\\
&\le \tfrac{2\sigma^2}{\mu (1-\alpha)}\tfrac{1}{m\alpha -(m-1)}\le \tfrac{4\sigma^2}{\mu(1-\alpha)}.
\end{align*}
Hence, similar to the proof of Theorem~\ref{main_ran_sc}, using the above relations and \eqref{constant_step_s_ran} in \eqref{main_bnd_sto}, we obtain
\begin{align*}
\bbe[Q(\underline{x}^k,x^*)]
&\le \tfrac{\alpha^k}{1-\alpha^k}\left[\Delta_{0,\sigma_0} + \tfrac{5\sigma^2}{\mu}\right],\\
\bbe[P(x^k,x^*)]
&\le2\alpha^k \left[\Delta_{0,\sigma_0}+ \tfrac{5\sigma^2}{\mu^2}\right], 
\end{align*}
where $\Delta_{0,\sigma_0}$ is defined in \eqref{def_DeltaS}. The second relation implies our results in \eqref{ran_bnd_s1_sto}. Moreover, \eqref{ran_bnd_s2_sto} follows from the same argument as we used in proving Theorem~\ref{main_ran_sc}.
% \end{proof}

\vgap
\setcounter{equation}{0}
\section{Concluding remarks}
%\todo{summary contributions}
In this paper, we propose a new randomized incremental gradient method, referred to as random gradient extrapolation method, for solving the classes of deterministic finite-sum optimization problems in \eqref{cp} and stochastic finite-sum optimization problems in \eqref{sp}, respectively. 
We demonstrate that without any exact gradient evaluation even at the initial point, this algorithm achieves optimal linear rate of convergence for deterministic strongly convex problems, as well as exhibiting optimal sublinear rate of convergence (up to a logarithmic factor) for stochastic strongly convex problems. 
All these complexity bounds have been established in terms of the total number of gradient computations of component function $f_i$ and the latter complexity bound on the computation of stochastic gradients is in fact asymptotically independent of the number of components $m$.
Moreover, we consider solving finite-sum problems in \eqref{cp} and \eqref{sp} in a distributed network setting with $m$ agents connected to a central server. 
Since each iteration of our proposed algorithm only involves constant number of communication rounds between the server and one randomly selected agent, it achieves linear communication complexity and avoids synchronous delays among agents.
It is worth pointing out that by exploiting the mini-batch technique, the algorithm can also achieve linear communication complexity for solving stochastic finite-sum problems, which is the best-known communication complexity for distributed stochastic optimization problems in the literature.

%%%%%%%%%%%%%%%%%%%%%%%%%%%%%%%%%%%%%%%%%%%%%%%%%%%%%%%%%%%%%%%%%%%
\bibliographystyle{plain}
\bibliography{glan-bib}
%%%%%%%%%%%%%%%%%%%%%%%%%%%%%%%%%%%%%%%%%%%%%%%%%%%%%%%%%%%%%%%%%%%
\end{document}